%% file: clean.tex
\documentclass[a4paper,11pt]{article}

\usepackage{graphicx}
\usepackage{amssymb}
\usepackage{amsmath}
\usepackage{booktabs}
\usepackage{algorithm}
\usepackage{algorithmic}
\usepackage{natbib}
\usepackage[T1]{fontenc} 
\usepackage{lmodern}
\usepackage{tikz}
\usepackage{pgfplots}

\newcommand{\field}[1]{\mathbb{#1}}  
\newcommand{\R}{\field{R}} 
\newcommand{\Zset}{\mathcal{Z}} 
\newcommand{\transp}{^{^{\intercal}}}
\newcommand{\onenorm}[1]{\left|\left|#1\right|\right|_1}
\newcommand{\twonorm}[1]{\left|\left|#1\right|\right|_2}
\newcommand{\infnorm}[1]{\left|\left|#1\right|\right|_\infty}
\newcommand{\knorm}[1]{\left|\left|#1\right|\right|_l}

\newenvironment{definition}[1][Definition]{\begin{trivlist}
\item[\hskip \labelsep {\bfseries #1}]}{\end{trivlist}}
\newtheorem{theorem}{Theorem}[section]
\newtheorem{lemma}[theorem]{Lemma}
\newenvironment{proof}[1][Proof]{\begin{trivlist}
\item[\hskip \labelsep {\bfseries #1}]}{\end{trivlist}}

\newcommand{\qed}{\nobreak \ifvmode \relax \else
      \ifdim\lastskip<1.5em \hskip-\lastskip
      \hskip1.5em plus0em minus0.5em \fi \nobreak
      \vrule height0.75em width0.5em depth0.25em\fi}

\hypersetup{
    colorlinks=true,
    citecolor=black,
    urlcolor=blue,
    linkcolor=black,
    pdfborder={0 0 0},
}

\begin{document}

\title{Robust counterparts of inequalities containing sums of maxima of linear functions}

\author{Bram L. Gorissen\footnote{corresponding author}, Dick den Hertog \\ \\
\textit{\small Tilburg University, Department of Econometrics and Operations Research, 5000 LE Tilburg, Netherlands} \\
\textit{\small {\tt \{b.l.gorissen,d.denhertog\}@tilburguniversity.edu}}}
\date{}
\maketitle

\begin{abstract}
This paper addresses the robust counterparts of optimization problems containing sums of maxima of linear functions. These problems include many practical problems, e.g.~problems with sums of absolute values, and arise when taking the robust counterpart of a linear inequality that is affine in the decision variables, affine in a parameter with box uncertainty, and affine in a parameter with general uncertainty.

In the literature, often the reformulation is used that is exact when there is no uncertainty. However, in robust optimization this reformulation gives an inferior solution and provides a pessimistic view. We observe that in many papers this conservatism is not mentioned. Some papers have recognized this problem, but existing solutions are either conservative or their performance for different uncertainty regions is not known, a comparison between them is not available, and they are restricted to specific problems. We describe techniques for general problems and compare them with numerical examples in inventory management, regression and brachytherapy. Based on these examples, we give recommendations for reducing the conservatism.
\end{abstract}

\begin{tikzpicture}[remember picture,overlay]
\node[anchor=south,yshift=10pt] at (current page.south) {\fbox{\parbox{\dimexpr\textwidth-\fboxsep-\fboxrule\relax}{\footnotesize This is an author-created, un-copyedited version of an article published in European Journal of Operational Research \href{http://dx.doi.org/10.1016/j.ejor.2012.10.007}{DOI:10.1016/j.ejor.2012.10.007}.}}};
\end{tikzpicture}
\section{Introduction}
\label{sec:intro}
Robust Optimization (RO) first appeared in \cite{Soyster1973}, and after receiving very little attention in the subsequent decades it has been an active research area since \cite{BenTal99} and \cite{Ghaoui97} started publishing new results in the late nineties. In its basic form, it requires a solution of an optimization problem to be feasible for any realization of the uncertain parameters in a given uncertainty set. For several choices of the uncertainty region this leads to tractable problems; for instance the robust counterpart (RC) of an LP with polyhedral uncertainty can be reformulated as an LP and the RC of an LP with ellipsoidal uncertainty is a conic quadratic program (CQP) \cite[p.~21]{BenTal}.

In RO two distinct formulations that are equivalent in the nonrobust case may have different RCs. This is the case for optimization problems containing the sum of maxima of linear functions, of which the sum of absolute values is a special case ($|x|=\max\{x,-x\}$). These problems arise in inventory management, supply chain management, regression models, tumor treatment, and many other practical situations.

Because RO is applied constraint-wise w.l.o.g.~and the objective function can always be formulated as a constraint \cite[p.~11]{BenTal}, we focus on the following robust constraint:
\begin{align}
  \ell(\zeta,x) + \sum_{i \in I} \max_{j \in J} \{ \ell_{ij}(\zeta,x) \} \leq d && \forall \zeta \in \Zset, \label{cons:general}
\end{align}
where $\ell$ and $\ell_{ij}$ are biaffine functions in the uncertain parameter $\zeta \in \R^L$ and the decision variable $x \in \R^n$, $d \in\R$ is the right hand side, and $\Zset \subset \R^L$ is a user-specified uncertainty region, e.g.~a box or an ellipsoid. We can assume this uncertainty region to be closed and convex w.l.o.g., as the left hand side is convex in $\zeta$ and consequently the worst case is always located at an extreme point of the uncertainty region. In the remainder of this paper our objective is always to minimize $d$, but all methods can still be applied when a different objective is used.

In the literature, often the following RC is used, obtained from a reformulation with analysis variables $y_i$ that is exact when there is no uncertainty:
\begin{align}
\textrm{(RC-R)} \qquad
  \ell(\zeta,x) + \sum_{i \in I} y_i &\leq d                  && \forall \zeta \in \Zset \label{cons:rc1} \\
                                 y_i &\geq \ell_{ij}(\zeta,x) && \forall \zeta \in \Zset \quad \forall i \in I \quad \forall j \in J. \label{cons:rc2}
\end{align}
In this formulation, to which we refer as the RC-R (RC of the reformulation), $y_i \in \R$ is a fixed variable, taking the worst case value of the $i^{th}$ term of the sum. In many cases the terms of the sum do not all reach their worst case value in the same realization of the uncertain parameter $\zeta$, and therefore the RC-R is a conservative reformulation of constraint \eqref{cons:general}, i.e.~a solution $(x,d)$ that is feasible for RC-R is also feasible for constraint \eqref{cons:general}, but not necessarily vice versa. However, this reformulation is frequently used without mentioning its conservatism. It has the advantage that the constraints are linear, so that tractable reformulations exist for many uncertainty regions.

\cite{BertsimasThiele2006} use the RC-R for a robust inventory problem which includes the sum of holding and backlogging costs: $\sum_{i=1}^n \max\{ c_h x_i, -c_b x_i \}$, where $x_i$ is the inventory level at time period $i$ and $c_h$ and $c_b$ are the holding and backlogging costs per time period respectively. Their uncertainty region is the intersection of $\{\zeta : \onenorm{\zeta_{1:i}} \leq \Gamma_i\} \quad \forall i$ and $\{\zeta : \infnorm{\zeta} \leq 1\}$, where $\zeta_{1:i}$ is the vector consisting of the first $i$ elements of $\zeta$, and $\Gamma_i$ is not necessarily integer (budgeted uncertainty). Their formulation allows that the values of different analysis variables take values based on different realizations of the uncertain parameter. The same uncertainty region is also used by \cite{Bertsimas2004} for a supply chain model, by \cite{thiele2004robust} for supply chains and revenue management, by \cite{JoseAlem2012139} for a production planning problem under uncertain demand, and by \cite{Wei2009} for a slightly more complicated inventory model in which items can be returned and remanufactured or disposed. With the exception of \cite[p.37]{thiele2004robust}, none of these papers mentions the conservatism of the formulation.

\cite{Ng2010557} treat a lot allocation problem where each order is assigned to one or more production locations before the production capacity of each location is fully known. Every production location is assigned to at most one order. Their uncertainty region is an ellipsoid over all production locations, while an analysis variable is used for every order, so their formulation is conservative.

\cite{Kropat2010} consider a robust linear cluster regression model using the sum of absolute values with polyhedral and ellipsoidal uncertainty regions. They introduce an analysis variable for every absolute value, which is conservative because every uncertain parameter affects two absolute values.

\cite{BenTal:2005:retailer} solve a multi-period inventory problem and use affine decision rules for the actual decisions (AARC). For every time period there are analysis variables indicating the costs in that time period. These costs are given by the maximum of holding costs and shortage costs, which both are linear functions of past demand. They note that an analysis variable should therefore be replaced with a function being the maximum of the two linear functions. Because that would lead to a very complicated robust counterpart, they replace the analysis variables with linear decision rules instead, which is conservative. Replacing analysis variables with linear decision rules was also done by \cite{BenTal:2009:inventory} and \cite{BenTal20111177}.

\cite{Bienstock2008389} were the first to identify and eliminate the conservatism of the RC-R. The idea behind their solution is that it suffices to make constraint  \eqref{cons:general} hold for just the vertices of the uncertainty region. The constraint then also holds for all other elements in the (convex) uncertainty region, because the constraint is convex in $\zeta$. They generalize it to a cutting plane method that only adds a subset of the vertices, which they successfully tested for computing basestock levels under budgeted uncertainty. It is not yet known how well this cutting plane method performs on other problems or different uncertainty regions.

In this paper, we make the following contributions. First, we identify that a conservative reformulation is often used in the literature without mentioning its conservatism. Second, we make a classification of solution approaches. Third, we present a new cutting plane method (Algorithm \ref{alg:iterrlc}). Fourth, we show that a linear constraint with biaffine uncertainty is a special case. This provides a method to solve or approximate a conic quadratic constraint with box uncertainty, for which currently no efficient methods are available. Fifth, we demonstrate all approaches on several numerical examples. Such a comparison is currently not available. Sixth, based on our numerical examples, we give recommendations for reformulating inequalities containing sums of maxima of linear functions.

The structure of this paper is as follows. Since it is not possible to compare the objective value of different formulations in RO, we introduce a new performance number which is independent of the reformulation in Section \ref{sec:truevalue}. In the examples we consider, this performance number is the slack in constraint \eqref{cons:general}. We provide an overview of exact (non-conservative) formulations, approximations, and cutting plane methods in Section \ref{sec:approaches}. In that section we show similarities between different methods, and also show how several approaches can be combined. The application scope of this paper is extended in Section \ref{sec:rclinwithnonaffineunc}, where we show that our methods can be applied to an uncertain conic quadratic constraint with box uncertainty. We evaluate the methods in Section \ref{sec:numericalexamples} on some small toy problems and three larger problems. We give conclusions in Section \ref{sec:conclusions}, and show how the robust counterparts in the aforementioned papers could be improved.

\section{The true robust value}
\label{sec:truevalue}
This section explains how two solutions from different reformulations can be compared. For a minimization problem containing sums of maxima of linear functions, we define the true robust value of a solution to be the maximum over $\zeta$ in $\Zset$ of the unreformulated problem with all decision variables fixed. We have assumed that in this paper our objective is always to minimize $d$, an analysis variable at the right hand side of \eqref{cons:general}, so in our case the true robust value of a solution $x$ is:
\begin{align}
  v_{true}(x) = \max_{\zeta \in \Zset} \left\{ \ell(\zeta,x) + \sum_{i \in I} \max_{j \in J} \{ \ell_{ij}(\zeta,x) \} \right\}. \label{true-value}
\end{align}
Determining this value is a difficult problem, because it requires the maximization of a convex function over a convex set. The global maximum over $\zeta \in \Zset$ of the sum of maxima of linear functions does not necessarily coincide with a maximum over $\zeta \in \Zset$ of one of those linear functions. One way to obtain the exact value is by solving the following optimization problem with integer variables for fixed $x$:
\begin{align}
& \max        \qquad && \ell(\zeta,x) + \sum_{i \in I} y_i && \label{vtrueasip} \\
& \mbox{s.t.} \qquad && y_i \leq \ell_{ij}(\zeta,x) + M(1-z_{ij})  && \forall i \in I \quad \forall j \in J  \notag \\
&                    && \sum_{j \in J} z_{ij} = 1                  && \forall i \in I \notag \\
&                    && \zeta \in \Zset, y \in \R^{|I|}, z \in \{0,1\}^{|I| \times |J|}, && \notag
\end{align}
where $M$ is a sufficiently large number. This problem is an MILP for polyhedral uncertainty, and an MIQCP for ellipsoidal uncertainty.

Another way of obtaining the exact value is by considering $|J|^{|I|}$ linear optimization problems, e.g.~for $I=\{1,2\}$ and $J=\{1,2\}$ consider the following four optimization problems:
\begin{align}
    &\max_{\zeta \in \Zset} \ell(\zeta,x) + \ell_{1,1}(\zeta,x) + \ell_{2,1}(\zeta,x), \label{vtrueasmanysimple} \\
    &\max_{\zeta \in \Zset} \ell(\zeta,x) + \ell_{1,1}(\zeta,x) + \ell_{2,2}(\zeta,x), \notag \\
    &\max_{\zeta \in \Zset} \ell(\zeta,x) + \ell_{1,2}(\zeta,x) + \ell_{2,1}(\zeta,x),\textrm{ and} \notag \\
    &\max_{\zeta \in \Zset} \ell(\zeta,x) + \ell_{1,2}(\zeta,x) + \ell_{2,2}(\zeta,x). \notag
\end{align}
Each of these problems is easy, because the maximum of an affine function over a box or an ellipsoid can be computed in a few operations. The true robust value is the largest value of the computed maxima.

If determining $v_{true}(x)$ is intractable, bounds can still be obtained. Filling in any $\zeta$ from the set $\Zset$ at the right hand side of equation \eqref{true-value}, for instance the nominal value, gives a lower bound. If the dimension of $\zeta$ is small, the bound can be improved with global optimization techniques. Upper bounds can be obtained by fixing $x$ in any conservative reformulation (such as those mentioned in Section \ref{sec:approx}) of constraint $\eqref{cons:general}$ and optimizing over the other variables.

\section{Solution approaches}
This section lists exact solution approaches and approximation methods, most of which can be applied to general RO problems containing the sum of maxima of linear functions. Many of these methods have been used before, but this full classification is new. This allows us to show the similary between some methods, and to show how approximations can be combined with exact methods.
\label{sec:approaches}
\subsection{Exact reformulations}
\subsubsection{Vertex enumeration}
Vertex enumeration is an exact solution method first used for a RO problem containing the sum of maxima by \cite{Bienstock2008389} that is powerful especially when the uncertainty region has a small number of vertices. Let $V$ denote the finite set of vertices, and consider the following reformulation of constraint \eqref{cons:general}:
\begin{align}
\textrm{(Vertex enumeration)} \qquad   \ell(\zeta,x) + \sum_{i \in I} y^\zeta_i &\leq d   && \forall \zeta \in V \label{cons:ve} \\
   y^\zeta_i &\geq \ell_{ij}(\zeta,x)                 && \forall i \in I \quad \forall j \in J \quad \forall \zeta \in V. \label{cons:ve2}
\end{align}
Constraint \eqref{cons:ve2} is no longer a semi-infinite constraint, because $V$ is a finite set. This reformulation is exact because the left hand side of constraint \eqref{cons:general} is convex in $\zeta$, and a convex function takes its maximum at an extreme point of its domain.

\subsubsection{Enumeration of robust linear constraints}
The RC-R is inexact because it has analysis variables that may take values corresponding to different worst case scenarios. A constraint with a single max function does not suffer from this problem, because an equivalent set of linear constraints can be formulated without analysis variables. An exact reformulation of RC \eqref{cons:general} can be obtained by first rewriting it as a constraint with a single $\max\{ \cdot \}$ function by enumerating all combinations, and then applying RO to the reformulation:
\begin{align*}
\textrm{(EORLC)} \qquad   &\ell(\zeta,x) + \sum_{i \in I} \ell_{i,j(i)}(\zeta,x) \leq d && \forall j(i) \in J \qquad \forall \zeta \in \Zset.
\end{align*}
We call this the enumeration of robust linear constraints (EORLC) formulation. It has the advantage that the constraints are linear, so that tractable reformulations exist for many uncertainty regions. For example, with $I=\{1,2\}$ and $J=\{1,2\}$, the EORLC formulation has the following constraints:
\begin{align*}
                    &\ell(\zeta,x) + \ell_{1,1}(\zeta,x)+\ell_{2,1}(\zeta,x) \leq d && \forall \zeta \in \Zset \\
                    &\ell(\zeta,x) + \ell_{1,1}(\zeta,x)+\ell_{2,2}(\zeta,x) \leq d && \forall \zeta \in \Zset \\
                    &\ell(\zeta,x) + \ell_{1,2}(\zeta,x)+\ell_{2,1}(\zeta,x) \leq d && \forall \zeta \in \Zset \\
                    &\ell(\zeta,x) + \ell_{1,2}(\zeta,x)+\ell_{2,2}(\zeta,x) \leq d && \forall \zeta \in \Zset.
\end{align*}
A similar formulation is also given by \cite{Bienstock2008389}, where it was neglected for its exponential size. While the number of constraints $|J|^{|I|}$ indeed grows exponentially with the number of terms in the summation, it is effective for small $|I|$. There are situations in which $I$ is small indeed, for instance with a planning horizon of up to ten periods.

EORLC has a strong relation to vertex enumeration that was not observed before to the best of our knowledge. In fact, EORLC is vertex enumeration on a different set. Constraint \eqref{cons:general} can be formulated as:
\begin{align*}
  \ell(\zeta,x) + \sum_{i \in I} \sum_{j \in J} \lambda_{ij} \ell_{ij}(\zeta,x) \leq d && \forall \lambda_i \in \Delta^{|J|-1} \quad \forall \zeta \in \Zset,
\end{align*}
where $\Delta^{|J|-1} = \{\lambda_i \in \R^{|J|} : \sum_{j \in J} \lambda_{ij} = 1, \lambda_{ij} \geq 0\}$ (the standard simplex in $\R^{|J|}$). The vertices of the simplex are given by unit vectors. If $\lambda_i$ is a unit vector, $\sum_{j \in J} \lambda_{ij} \ell_{ij}(\zeta,x)$ simplifies to a single $\ell_{ij}(\zeta,x)$. It follows that vertex enumeration on the $|I|$ simplices gives the EORLC reformulation.

EORLC can benefit from two preprocessing steps in order to reduce the final number of constraints. First, every $\max\{\cdot\}$ term should contain at most one function that does not depend on $\zeta$. If it has more than one, those functions can all be replaced with a single analysis variable. Second, inequalities of the form $a \max\{0,\ell(\zeta,x) \} + b \max\{ 0, -\ell(\zeta,x) \} \leq d$ with $a,b>0$, which are often used for holding and backlogging costs, should be reformulated as $\max\{a \ell(\zeta,x), -b \ell(\zeta,x) \} \leq d$.

\subsubsection{Cases with special structure}
\label{sec:specialstructure}
There are several special cases of \eqref{cons:general} that allow an exact reformulation. This section lists a few general cases. More specific cases can be found in \cite[Ch. 12.2]{BenTal} and in \cite{Xu:2009}.

The first case is when the uncertainty region $\Zset$ is the direct product of sets $\Zset_i$, where term $i$ in the left hand side of constraint \eqref{cons:general} is only affected by $\Zset_i$. The RC-R is exact because all analysis variables can take their worst case values simultaneously.

The second case is when the inequality contains the sum of absolute values of linear functions of $\zeta_i$ and $x$:
\begin{align}
	\ell(\zeta,x) + \sum_{i \in I} | \alpha_i(x) + \beta_i(x) \transp \zeta | \leq d \qquad \forall \zeta \in \Zset, \label{eq:special1}
\end{align}
where $\alpha_i : \R^n \to \R$ and $\beta_i : \R^n \to \R$ are linear functions, the components of $\beta_i$ that may be nonzero for one $i$ are zero for all other $i$ in $I$, and the uncertainty region $\Zset$ is centrosymmetric around $\zeta=0$, i.e.~$\Zset$ is closed under changing the sign of one or more vector elements (a box and an ellipsoid are examples of such sets). Note that the assumption that the symmetry is around $0$ is made w.l.o.g. The following constraint is equivalent to \eqref{eq:special1}:
\begin{align}
\ell(\zeta,x) + \sum_{i \in I} \left\{ | \alpha_i(x) | + \beta_i(x) \zeta_i \right\} \leq d \qquad \forall \zeta \in \Zset. \label{eq:special2}
\end{align}
Equivalence is readily checked by conditioning on the sign of $\alpha_i(x)$. The formulation for \eqref{eq:special2} where absolute values are replaced with analysis variables is equivalent to \eqref{eq:special2}, since the analysis variables do not depend on $\zeta$.

The third case is when, for a fixed $i \in I$, each of linear functions under the maximum is the sum of a common nonnegative linear function of $\zeta$ and a linear function of $x$:
 \[   \ell(\zeta,x) + \sum_{i \in I} \max_{j \in J} \{ \alpha_i(\zeta) + \beta_i(\zeta) \ell_{ij}(x) \} \leq d \qquad \forall \zeta \in \Zset,    \]
where $\alpha_i : \R^L \to \R$ and $\beta_i : \R^L \to \R_+$ are linear functions. The common functions of $\zeta$ can be placed outside the $\max \{\cdot \}$ expression:
 \[   \ell(\zeta,x) + \sum_{i \in I} \left[ \alpha_i(\zeta) + \beta_i(\zeta) \max_{j \in J} \{ \ell_{ij}(x) \} \right] \leq d \qquad \forall \zeta \in \Zset,    \]
and the RC-R of this constraint is exact. If the range of $\beta_i$ is $\R$ instead of $\R_+$, then $\max_{j \in J} \{ \ell_{ij}(x) \}$ should be $\min_{j \in J} \{ \ell_{ij}(x) \}$ when $\beta_i(\zeta)<0$. This can be modeled as follows. Given a subset $I_+ \subseteq I$, we define a set which consists of those $\zeta$ for which $\beta_i(\zeta) \geq 0$ for $i \in I_+$, and $\beta_i(\zeta) \leq 0$ for $i \in I \backslash I_+$:
\[ \Zset(I_+) = \Zset \cap \{\zeta : \beta_i(\zeta) \geq 0 \quad \forall i \in I_+, \quad \beta_i(\zeta) \leq 0 \quad \forall i \in I \backslash I_+ \}. \]
Note that \[ \Zset = \bigcup_{I_+ \subseteq I } \Zset(I_+). \]
The constraint can now be written as:
\[ \ell(\zeta,x) + \sum_{i \in I} \alpha_i(\zeta) + \sum_{i \in I_+} \beta_i(\zeta) \max_{j \in J} \{ \ell_{ij}(x) \} + \sum_{i \in I \backslash I_+} \beta_i(\zeta) \min_{j \in J} \{ \ell_{ij}(x) \} \leq d \qquad \forall \zeta \in \Zset(I_+) \qquad \forall I_+ \subseteq I. \]
The number of constraints is $2^{|I|}$ (one for each $I_+ \subseteq I$), which is less than the $|J|^{|I|}$ constraints obtained with EORLC. The $\max$ and $\min$ expressions do not depend on $\zeta$, so the reformulation with analysis variables is exact. Each constraint is still convex despite the min expressions, because their coefficients $\beta_i(\zeta)$ are negative.

\subsection{Conservative approximations}
\label{sec:approx}
The RC-R \eqref{cons:rc1}-\eqref{cons:rc2} is a conservative approximation to \eqref{cons:general}. We discuss how the conservatism can be decreased, and also present a new method.

The variables $y_i$ in the RC-R have been introduced for modeling the $\max \{\cdot \}$ expression. We do not need to know their values because they do not correspond to a ``here and now'' decision. Only $x$ has to be known for implementing a solution. The values of $y_i$ may be adjusted according to the realization of the uncertain parameter $\zeta$ as long as the constraints hold for every realization of $\zeta$ in the perturbation set (Adjustable RC). As first applied by \cite{BenTal:2005:retailer} for a multi-stage problem and later also done by \cite{BenTal:2009:inventory} and \cite{BenTal20111177}, we can make $y_i$ an affine function of $\zeta$, which leads to the Affinely Adjustable RC of the reformulation (AARC-R). After substituting $y_i = v_i + w_i \transp \zeta$ (with decision variables $v_i \in \R$ and $w_i \in \R^L$), constraints \eqref{cons:rc1}-\eqref{cons:rc2} become:
\begin{align*}
\textrm{(AARC-R)} \qquad \ell(\zeta,x) + \sum_{i \in I} \left( v_i + w_i \transp \zeta \right) &\leq d && \forall \zeta \in \Zset \\
 v_i + w_i \transp \zeta  &\geq \ell_{ij}(\zeta,x) && \forall \zeta \in \Zset \quad \forall i \in I \quad \forall j \in J.
\end{align*}
This substitution gives a less conservative reformulation, while the robust counterpart is often still tractable, because robust linear constraints are tractable for a wide class of uncertainty regions. The power of the AARC-R can be increased by lifting the uncertainty region to a higher dimension \citep{Chen2009}. If $\ell_{ij}$ does not depend on one or more components of $\zeta$ for all $j$ for some fixed $i$, the computational complexity can seemingly be reduced by making $y_i$ a function of only those components of $\zeta$ that appear in term $i$ for one or more $j$. However, it is easy to construct an example where this reduction introduces more conservatism.

There are two different approaches that also lead to the formulation AARC-R. The first approach is considering the Fenchel dual problem of maximizing the left hand side of constraint \eqref{cons:general} over $\zeta$ in $\Zset$. We give the full derivation and a proof of equivalence to the AARC-R in Appendix \ref{sec:rewrite-fenchel}. The other approach is derived in Appendix \ref{sec:deriv-marleen}.

An approach that is less conservative than an affine decision rule, is a quadratic decision rule:
 \[  y_i = v_i + w_i \transp \zeta + \zeta \transp W_i \zeta,  \]
where $v_i \in \R$, $w_i \in \R^L$, and $W_i \in \R^{L \times L}$ are new analysis variables. This is called a Quadratically Adjustable RC of the reformulation (QARC-R) which is known to be tractable for ellipsoidal uncertainty and (under some restrictions) for box uncertainty. A deterministic reformulation of the QARC-R with ellipsoidal uncertainty is given in Appendix \ref{sec:deriv-qarc-ellipsoidal}, resulting in $|I||J|+1$ LMIs of size $|L+1|$. For box uncertainty, when the quadratic terms are restricted such that each element of $\zeta$ is multiplied with itself and at most one other element of $\zeta$, we can use the result by \cite{Yanikoglu2012} to write the RC as an SDP with $(|I||J|+1)\lceil L/2 \rceil$ variable matrices of size three.

A new way to reduce the conservatism of the RC-R is by first combining several $\max$ expressions before reformulating. In order to do this, partition $I$ into $|G|$ groups: $I = \bigcup_{g \in G} I_g$, where the groups are mutually disjoint: $I_{g_1} \cap I_{g_2} = \emptyset$ for $g_1,g_2 \in G, g_1 \neq g_2$, and partition in such a way that $|I_g|$ is small for all $g$. Now introduce analysis variables for all $g$:
\begin{align*}
\ell(\zeta,x) + \sum_{g \in G} y_g &\leq d && \forall \zeta \in \Zset \\
y_g                                &\geq \sum_{i \in I_g} \max\{ \ell_{ij}(\zeta,x) \} && \forall \zeta \in \Zset \quad \forall g \in G.
\end{align*}
Because the cardinality of $I_g$ is small, the sum of $\max$ expression within each constraint can be transformed into a single $\max$ (EORLC). Each constraint in this reformulation is therefore tractable and can be solved exactly, but conservatism still comes from the analysis variables $y_g$. These analysis variables can also be written as a linear, quadratic or more general function of $\zeta$.

\subsection{Cutting plane methods}
In this section we describe two cutting plane methods. The first one is based on vertex enumeration and was used in \cite{Bienstock2008389}, the other one is new and based on EORLC.

Vertex enumeration results in very large problems if there are many vertices. For box uncertainty, the number of vertices grows exponentially in the time horizon. Also for budgeted uncertainty, which is described in Section \ref{sec:intro}, the number of vertices quickly becomes very large. The cutting plane method outlined in Algorithm \ref{alg:itervertex} adds only a subset of the vertices.
\begin{algorithm}[H]
\caption{Cutting plane method based on vertex enumeration}
\label{alg:itervertex}
\begin{algorithmic}[1]
\REQUIRE A linear program $LP$ with constraints $\eqref{cons:ve}$
\STATE $V := \{\zeta^{nom}\}$ (the nominal value)
\STATE $k := 0$
\REPEAT
  \STATE $k := k+1$
  \STATE Solve $LP$
  \STATE Let $x^*$ be the minimizer of $LP$, and $LB$ be the value of $d$ in $LP$
  \STATE Let $UB := v_{true}(x^*)$, and $\zeta^k$ be its maximizer
  \STATE $V := V \cup \{ \zeta^k \}$ \label{algl:addtootheralg}
\UNTIL{$UB-LB < \varepsilon$}
\end{algorithmic}
\end{algorithm}

While the algorithm is running, $LB$ is a lower bound on the optimal value of $d$ in \eqref{cons:ve} because it is the value of a relaxation of the original problem, and $UB$ is an upper bound on the optimal $d$, because it is the maximum for some $x$ and not for the optimal $x$. The difference $UB-LB$ indicates the current violation of the constraint. If $\varepsilon$ is set to a larger value, the algorithm terminates quicker but does not give the optimal solution. If $\varepsilon=0$ and the algorithm terminates, the final solution is robust feasible and robust optimal. It is possible that this algorithm still enumerates all constraints, but we have not encountered problems for which this is the case. \cite{Bienstock2008389} find that for their problem the number of iterations does not increase when the time horizon is increased, and is around four on average even for 150 time periods. The number of vertices of their uncertainty region cannot be deducted from their report. In our numerical examples we often find a larger number of iterations. The same cutting plane method was used by \cite{Bohle2010245}, again for budgeted uncertainty, but solution times and the number of generated constraints are not reported. It is still an open question how well this method works on other problems, and for non-polyhedral uncertainty regions.

The most time consuming step is determining the true value, since this involves the maximization of a convex function. As pointed out by \cite{Bienstock2008389}, it is not necessary to find the optimal solution. It suffices to find any $\zeta$ for which the maximization problem has a larger objective value than $LB$, because it corresponds to a violated constraint in the relaxation, and adding that constraint strengthens the relaxation. Bienstock et al. report that also the problem $LP$ does not need to be solved to optimality. Optimization can stop as soon as the objective value is less than $UB$, because at that point the current solution is already better than the solution in the previous run.

The second cutting plane method is new and based on EORLC. If $|I|$ becomes too large for EORLC, this cutting plane method adds only a subset of the robust linear constraints. It starts with the nominal problem and adds new robust constraints until the solution is robust feasible (Algorithm \ref{alg:iterrlc}).
\begin{algorithm}[H]
\caption{Cutting plane method based on EORLC}
\label{alg:iterrlc}
\begin{algorithmic}[1]
\REQUIRE A linear program $LP$ with constraints $\eqref{cons:ve}$
\STATE $V := \{\zeta^{nom}\}$ (the nominal value)
\STATE $k := 0$
\REPEAT
  \STATE $k := k+1$
  \STATE Solve $LP$
  \STATE Let $x^*$ be the minimizer of $LP$, and $LB$ be the value of $d$ in $LP$
  \STATE Let $UB := v_{true}(x^*)$, and $\zeta^k$ be its maximizer
  \STATE Let $j^k_i$ be the maximizer of $\max_{j \in J} \ell_{ij}(\zeta^k,x)$
  \STATE Add the following robust constraint to $LP$:
\begin{align*}
\ell(\zeta,x) + \sum_{i \in I} \ell_{i, j^k_i}(\zeta,x) \leq d \qquad \forall \zeta \in \Zset
\end{align*} \label{algl:insertpoint}
\UNTIL{$UB-LB < \varepsilon$} \label{algl:afterinsertpoint}
\end{algorithmic}
\end{algorithm}

Similar to Algorithm \ref{alg:itervertex}, a lower and upper bound are reported and the stopping criterion can be adjusted. Also for this algorithm, we have not encountered numerical examples in which all constraints are enumerated. The main difference between this algorithm and Algorithm \ref{alg:itervertex} is the constraint that is added in every iteration.

Just as with the cutting plane method based on vertex enumeration, determining the true value is the most time consuming step, and also here the optimization problem for determining the true value does not need to be solved to optimality as long as the value is larger than $LB$. Also $LP$ does not need to be solved to optimality.

Both cutting plane methods can easily be combined so that two sets of constraints are added per iteration. This is done by inserting line \ref{algl:addtootheralg} of Algorithm \ref{alg:itervertex} in between lines \ref{algl:insertpoint} and \ref{algl:afterinsertpoint} of Algorithm \ref{alg:iterrlc}.

The stopping criterion in both algorithms is $UB-LB < \varepsilon$. It is also possible to use a dimensionless stopping criterion such as $2 (UB-LB) / (1+|UB+LB|) < \varepsilon$. This stopping criterion is based on the relative gap.

\section{RC of a linear constraint with biaffine uncertainty}
\label{sec:rclinwithnonaffineunc}
Constraint \eqref{cons:general} may appear in RO itself when a constraint has biaffine uncertainty, and the uncertainty region of one parameter is a box. In general, such a constraint can be written as:
\begin{align} \label{eq:doubleuncertainty1}
\tilde{\ell}(\zeta^{(1)},\zeta^{(2)},x)  \leq d   \qquad \forall \zeta^{(1)} : \infnorm{\zeta^{(1)}} \leq \rho
                                  \quad  \forall \zeta^{(2)} \in \Zset,
\end{align}
where $\tilde{\ell} : \R^{|I|} \times \R^{L} \times \R^n \to \R$ is a triaffine function, $\rho$ is the radius of the box, and $\Zset$ is the uncertainty region of $\zeta^{(2)}$. Constraints with biaffine uncertainty have never been investigated before to the best of our knowledge. To show the equivalence to constraint \eqref{cons:general}, choose $\ell$, $\ell_{i}$ and $I$ in such a way that constraint \eqref{eq:doubleuncertainty1} is equivalent to:
\begin{align*}
\ell(\zeta^{(2)},x) + \sum_{i \in I} \zeta^{(1)}_i \ell_{i}(\zeta^{(2)},x) \leq d   \qquad \forall \zeta^{(1)} : \infnorm{\zeta^{(1)}} \leq \rho
                                                                                    \quad  \forall \zeta^{(2)} \in \Zset,
\end{align*}
and maximize the left hand side with respect to $\zeta^{(1)}$:
\begin{align}
\ell(\zeta^{(2)},x) + \sum_{i \in I} \rho | \ell_{i}(\zeta^{(2)},x) | \leq d  \qquad \forall \zeta^{(2)} \in \Zset, \label{eq:doubleuncertainty2}
\end{align}
which is indeed in the form of constraint \eqref{cons:general}. This RC includes many practical problems, of which we give four examples.

The first example is a constraint $a \transp x \leq d$ where each element of $a$ has been estimated using a regression model with the same regressors $\zeta$ for every element: $a_i = \beta^0_i + (\beta^{(i)}) \transp \zeta + \varepsilon_i$ ($\beta^0_i \in \R, \beta^{(i)} \in \R^L, \varepsilon_i \in \R$), in that case the constraint is:
\[ (\beta^0 + B \zeta + \varepsilon) \transp x \leq d. \]
We assume that the coefficients $\beta^0 \in \R^n$ and $B \in\R^{n \times L}$, the matrix having $\beta^{(i)}$ as rows, have been estimated with some error, and also that the regressors $\zeta$ are not fully known. The constraint can be written as constraint \eqref{eq:doubleuncertainty1} if $\beta^0$, $B$, $\zeta$ and $\varepsilon$ lie in some specified uncertainty region, and the uncertainty of either $B$ or $\zeta$ is a box.

The second example appears in Adjustable RCs. Consider a robust constraint:
\[   \ell(\zeta^{(1)},\zeta^{(2)},x) + b \transp y \leq d \qquad \forall \zeta^{(1)} : \infnorm{\zeta^{(1)}} \leq \rho \quad \forall \zeta^{(2)} \in \Zset,   \]
where $y \in \R^{m_1}$ is an adjustable variable that represents a ``wait and see'' decision that can be made after $\zeta^{(1)}$ and $\zeta^{(2)}$ are (partially) observed, and $b$ is a vector of known coefficients (fixed recourse). Determining the true optimal policy for $y$ as a function of $\zeta^{(1)}$ and $\zeta^{(2)}$ is often intractable, which is why a suboptimal $y$ is often determined by a parameterization, i.e.~$y$ is written as a function of which the coefficients are decision variables. The first parameterization in RO was proposed by \cite{BenTal:2005:aarc} who proposed an affine decision rule, which results in a problem which is in the same class (LP, CQP or SDP, depending on the uncertainty region) as the problem with $y$ as a ``here and now'' decision variable. This was later extended to a quadratic decision rule with an ellipsoidal uncertainty region, resulting in an SDP (\cite{BenTal}), and to a polynomial decision rule of arbitrarily large degree restricted to uncertainty regions described by polynomial inequalities, resulting in a conservative SDP (\cite{BertsimasPolynomials}). The latter includes the biaffine decision rule $y = \ell^{'}(\zeta^{(1)},\zeta^{(2)},v)$, where $\ell^{'}$ is a function $\R^{|I|} \times \R^{L} \times \R^{m_2} \to \R^{m_1}$, and $v$ is a vector of coefficients to be determined by the model. This decision rule could be very useful if a problem is affected by several sources of uncertainty. The advantage over affine decision rules is that the former includes cross terms of $\zeta^{(1)}$ and $\zeta^{(2)}$. Applying the results of \cite{BertsimasPolynomials} gives an SDP which is not only conservative, but also has a potentially large instance size. E.g.~if $|I|$ (the dimension of $\zeta^{(1)}$) is very small but $\zeta^{(2)}$ is in a box of (large) dimension $L$, the conservative SDP has at least one variable matrix of size $(|I|+1)L+|I|+1$ and over $2|L|$ smaller matrices, while our result is exact and gives a practically solvable LP. If instead $\zeta^{(2)}$ is in an ellipsoid of (large) dimension $L$, the SDP still has the large variable matrix of size $(|I|+1)L+|I|+1$ while our exact result gives a CQP for which efficient solvers are available.

The third example is a constraint with unknown coefficients and implementation error. Consider the following constraint:
\begin{align}
	\ell(\zeta^{(1)}) \transp x \leq d   \qquad \forall \zeta^{(1)} \in \Zset_1, \label{eq:impl-error}
\end{align}
where $\ell$ is a vector of $n$ functions that are linear in the uncertain parameter $\zeta^{(1)}$. Now suppose that there is implementation error in $x$, i.e.~instead of $x$ we implement a vector of which component $i$ is given by $x_i + \zeta_i^{(2)}$ (additive implementation error) or by $\zeta_i^{(2)} x_i$ (multiplicative implementation error). After substituting $x$, constraint \eqref{eq:impl-error} becomes:
\[ \ell(\zeta^{(1)}) \transp (x + \zeta^{(2)}) \leq d   \qquad \forall \zeta^{(1)} \in \Zset_1 \quad \forall \zeta^{(2)} \in \Zset_2, \]
in case of additive implementation error, and:
\[ \sum_{i=1}^n \ell_i(\zeta^{(1)}) \zeta_i^{(2)} x_i \leq d   \qquad \forall \zeta^{(1)} \in \Zset_1 \quad \forall \zeta^{(2)} \in \Zset_2, \]
in case of multiplicative implementation error. Both constraints are special cases of constraint \eqref{eq:doubleuncertainty1} if either $\Zset_1$ or $\Zset_2$ is a box.

The fourth example is the following robust constraint with box uncertainty:
\begin{align}\label{eq:rcqwithbox}
\ell(\zeta^{(1)},x) + \sum_{k \in K} \knorm{ \ell_{k}(\zeta^{(1)},x)} \leq d  \qquad \forall \zeta^{(1)} : \infnorm{\zeta^{(1)}} \leq \rho, 
\end{align}
where $\ell_{k}$ is a vector of $L$ linear functions, and $l$ equals $1$, $2$ or $\infty$. In order to see that this constraint is equivalent to constraint \eqref{eq:doubleuncertainty1}, note that constraint \eqref{eq:rcqwithbox} is a reformulation of the following constraint:
\begin{align*}
\ell(\zeta^{(1)},x) + \sum_{k \in K} (\zeta_{k}^{(2)}) \transp \ell_{k}(\zeta^{(1)},x) \leq d
                                                            \qquad \forall \zeta^{(1)} : \infnorm{\zeta^{(1)}} \leq \rho
                                                            \qquad \forall \zeta^{(2)} : \knorm{\zeta_{k}^{(2)}}^* \leq 1 \quad \forall k \in K,
\end{align*}
where $\knorm{\cdot}^*$ is the dual norm, and $\zeta_k^{(2)}$ is a vector in $\R^L$ for all $k$ in $K$. We will show how the results in this paper can be used for solving the robust constraint \eqref{eq:rcqwithbox} for different choices of $l$.

For $l=2$ and $|K|=1$, constraint \eqref{eq:rcqwithbox} is a robust conic quadratic constraint, for which the RC is known only in special cases, one of which is when the vertices of the uncertainty region can be enumerated \cite[p.~159]{BenTal}. The case $|K|>1$ has not been covered yet. A (conservative) reformulation with analysis variables for every $\twonorm{ \ell_{k}(\zeta^{(1)},x)}$ reduces the problem to a problem with one linear constraint and $|K|$ robust conic quadratic constraints, each of which is of the form \eqref{eq:rcqwithbox} with $|K|=1$, so it can be reformulated using vertex enumeration. A different approach, that is not only exact but also results in a smaller problem than obtained by using analysis variables, is to apply vertex enumeration to constraint \eqref{eq:rcqwithbox} directly. Vertex enumeration is exact because the constraint is convex in $\zeta^{(1)}$. If the dimension of the box is very large, (iterative) vertex enumeration is no longer tractable, so even the case $|K|=1$ becomes unsolvable, and tractable conservative reformulations are not known. Our reformulation \eqref{eq:doubleuncertainty2} allows the use of all approaches from Section \ref{sec:approaches}.

For $l=1$ constraint \eqref{eq:rcqwithbox} is a special case of constraint \eqref{cons:general}, and we know from Section \ref{sec:approaches} how to solve it or how to find a conservative reformulation. The reformulation may be useful if it is solved with the RC-R, the AARC-R or the QARC-R, because the resulting formulation is very different. Vertex enumeration and EORLC are not useful on the reformulation \eqref{eq:doubleuncertainty2}, because they correspond with EORLC and vertex enumeration on the original constraint, respectively.

For $l=\infty$ constraint \eqref{eq:rcqwithbox} is a special case of constraint \eqref{cons:general} with $I = K$ and $J = \{1,2,..,L\}$, and we know from Section \ref{sec:approaches} how to solve it or how to find a conservative reformulation. The reformulation \eqref{eq:doubleuncertainty2} may be useful if it solved with the RC-R, the AARC-R or the QARC-R, because the resulting formulation is very different. The reformulation allows to do vertex enumeration on $\zeta^{(2)}$, which may be faster than vertex enumeration in the original constraint (which is done on the vertices of $\zeta^{(1)}$) if $L$ and $|K|$ are small relative to $|I|$.

For all three cases of $l$, it holds that when EORLC is used on the reformulation \eqref{eq:doubleuncertainty2}, the resulting constraints are the same ones resulting from vertex enumeration on constraint \eqref{eq:rcqwithbox}. This implies that reformulating \eqref{eq:doubleuncertainty2} is a redundant step if it is followed by EORLC.
\section{Numerical examples}
\label{sec:numericalexamples}
The LP, MILP, CQP and MICQP problems in this paper have been solved with AIMMS 3.11 FR3 x32 with ILOG CPLEX 12.1 unless stated otherwise. SDP problems have been modeled with CVX (\cite{CVX}) and solved with SDPT3 (YALMIP (\cite{Lofberg:2010}) would nowadays be a better choice as it allows the user to specify a linear constraint with quadratic uncertainty with an ellipsoidal uncertainty region directly). Reported computing times have been obtained under Windows XP SP3 on an Intel Core2 Duo E6400 (2.13 GHz) processor and 2 GB of RAM.
\subsection{Computing the true robust value}
We have run some experiments to determine how quickly the true robust value can be computed using the two different methods from Section \ref{sec:truevalue}. Using the problem with integer variables \eqref{vtrueasip}, the worst case $\zeta$ can be determined in less than a minute for $|I|=50$ and $|J|=3$ for box uncertainty of dimension $50$, and in $7$ seconds for $|I|=20$ and $|J|=2$ for ellipsoidal uncertainty of dimension $20$. For testing the performance of maximizing the $|J|^{|I|}$ linear optimization problems \eqref{vtrueasmanysimple}, we have created a single threaded C++ application that creates an affine function with coefficients randomly taken from the interval $[-100,100]$, maximizes that function, randomly selects new coefficients, et cetera. Only the running time of the maximization step is measured. It can maximize $1.5 \times 10^8$ affine functions $f : \R^{50} \to \R$ per minute over a box, which is $10^{23}$ times slower than the MILP. It can maximize $3 \times 10^8$ affine functions $g : \R^{20} \to \R$ per minute over an ellipsoid, which is 32 times faster than the MIQCP.
\subsection{Illustrative small problems}
\label{sec:toyproblems}
Consider the following toy optimization problem:
\begin{align*}
\textrm{(TOY1)} \qquad
& \min  			\qquad	&&d && \\
& \mbox{s.t.} \qquad	&&d \geq \max\{x, x + \zeta\} + \max\{x, x-\zeta\} && \forall \zeta \in [-1,1] \\
&                     &&d \in \R, x \in \R_+. &&
\end{align*}
The optimal value of this problem is $1$ ($x=0, d=1$). If we model this problem as an LP and then apply RO, we get the following model:
\begin{align*}
\textrm{(RC-R)} \qquad
&\min        \qquad	 &&y_1 + y_2           && \\
&\mbox{s.t.} \qquad	 &&y_1 \geq  x         && \forall \zeta \in [-1,1] \\
&                    &&y_1 \geq  x + \zeta && \forall \zeta \in [-1,1] \\
&                    &&y_2 \geq  x         && \forall \zeta \in [-1,1] \\
&                    &&y_2 \geq  x - \zeta && \forall \zeta \in [-1,1] \\
&                    &&y_1,y_2 \in \R, x \in \R_+. &&
\end{align*}
The optimal value of this model is $2$ ($x=0, y_1 = y_2 = 1$), so the RC-R is not exact. If we substitute $y_i = v_i + w_i \zeta$, the model becomes:
\begin{align*}
\mbox{(AARC-R)} \qquad
&\min        \qquad  &&y && \\
&\mbox{s.t.} \qquad  &&y \geq v_1 + w_1 \zeta + v_2 + w_2 \zeta   && \forall \zeta \in [-1,1]  \\
&                    &&v_1 + w_1 \zeta \geq x                     && \forall \zeta \in [-1,1] \\
&                    &&v_1 + w_1 \zeta \geq x + \zeta             && \forall \zeta \in [-1,1] \\
&                    &&v_2 + w_2 \zeta \geq x                     && \forall \zeta \in [-1,1] \\
&                    &&v_2 + w_2 \zeta \geq x - \zeta             && \forall \zeta \in [-1,1] \\
&                    &&v_1,v_2,w_1,w_2,y \in \R, x \in \R_+.      &&
\end{align*}
The optimal value of this problem is $1$ ($v_1=v_2=\frac{1}{2}, w_1=\frac{1}{2}, w_2=-\frac{1}{2}, y = 1$), which is the same as the optimal value of the original problem. So, in this case the AARC-R closes the gap. This is not always the case as the following example shows:
\begin{align*}
\mbox{(TOY 2)} \qquad
&\min  			\qquad	&& d && \\
&\mbox{s.t.} \qquad	&& d \geq \max\{x, x + \zeta_1 + \zeta_2 \} + \max\{x, x + \zeta_1 - \zeta_2\} && \nonumber \\
&                   && \qquad + \max\{x, x - \zeta_1 + \zeta_2\} + \max\{x, x - \zeta_1 - \zeta_2\} && \forall \zeta \in [-1,1]^2 \\
&                   && d \in \R, x \in \R_+. && 
\end{align*}
We obtain $v_{true}=2$, $v_{RC-R}=8$ and $v_{AARC-R}=4$. The AARC-R is an improvement over the RC-R, but is not exact. The QARC-R, which can be derived by reparameterizing the uncertainty region to $[0,1]^2$, applying the work of \cite{Yanikoglu2012} to obtain an LMI description of the uncertainty region, and formulating the RC to a deterministic program by dualization, gives $v_{QARC-R} = 2.83$. While this value is much closer to $2$ than the value of the AARC-R, it is still inexact.

\subsection{Least absolute deviations regression with errors-in-variables}
\label{sec:exreg}
An errors-in-variables regression model is a model $y_i = \beta_0 + \beta_1 x_i^* + \varepsilon_i$ ($\varepsilon_i \sim N(0,\sigma^2)$ i.i.d.) where $x_i^*$ cannot be measured accurately. Only $x_i$ and $y_i$ with $x_i^* = (1+\zeta_i)x_i$ are observed, where $\zeta_i$ is an unknown measurement error. The least absolute deviations approach estimates $\beta_0$ and $\beta_1$ by minimizing $\sum_{i=1}^n |y_i - \beta_0 - \beta_1 x_i^*|$:
\[ \min_{\beta_0, \beta_1} \sum_{i=1}^n \max\{y_i - \beta_0 - \beta_1 (1+\zeta_i) x_i, \beta_0 + \beta_1 (1+\zeta_i) x_i - y_i \}.     \]
Because we do not know the values $\zeta_i$, we can apply RO:
\begin{align*}
\min_{\beta_0, \beta_1} \max_{\zeta \in \Zset} \sum_{i=1}^n \max\{y_i - \beta_0 - \beta_1 (1+\zeta_i) x_i, \beta_0 + \beta_1 (1+\zeta_i) x_i - y_i \},
\end{align*}
where we choose the uncertainty region $\Zset$ to be ellipsoidal: $\Zset = \{ \zeta \in \R^n : \twonorm{ \zeta } \leq \Omega \}$. Note that this is the second special cases in Section \ref{sec:specialstructure} because it is the sum of absolute values, $\zeta_i$ appears only in term $i$, and the uncertainty region is centrosymmetric around $\zeta=0$. It follows that the RC can be written as:
\begin{align}
\min_{\beta_0, \beta_1} \sum_{i=1}^n \max\{y_i - \beta_0 - \beta_1 x_i, \beta_0 + \beta_1 x_i - y_i \} + \Omega \twonorm{\beta_1 x}, \label{ex-reg-true}
\end{align}
which can be reformulated as a CQP. Even though there is this explicit and simple formulation of the exact RC, we also try the other approaches from Section \ref{sec:approaches} for a comparison. The RC-R is as follows:
\begin{align*}
\mbox{(RC-R)} \qquad
\min  			\qquad &  \sum_{i=1}^n z_i  \\
\mbox{s.t.} \qquad &  z_i \geq  y_i - \beta_0 - \beta_1 x_i + \Omega |\beta_1 x_i| && \forall i \in I \\
                   &  z_i \geq  \beta_0 + \beta_1 x_i - y_i + \Omega |\beta_1 x_i| && \forall i \in I,
\end{align*}
which for comparison with formulation \eqref{ex-reg-true} can also be written as:
\begin{align*}
\min_{\beta_0, \beta_1} \sum_{i=1}^n \max\{y_i - \beta_0 - \beta_1 x_i, \beta_0 + \beta_1 x_i - y_i \} + \Omega \onenorm{\beta_1 x}.
\end{align*}
This is more pessimistic than constraint \eqref{ex-reg-true}, since $\onenorm{\cdot} \geq \twonorm{\cdot}$. For the RC-R, the worst case realization in the ellipsoidal set has all but one elements equal to 0. Therefore, it may seem that the RC-R is the RC of a linear constraint with interval uncertainty, but this is not the case. The AARC-R is given by:
\begin{align*}
\min  			\qquad	&  d                                                                     && \\
\mbox{s.t.} \qquad	&  d \geq \sum_{i=1}^n v_i + w_i \transp \zeta                           && \forall \zeta \in \R^n : \twonorm{\zeta} \leq \Omega \\
                    &  v_i + w_i \transp \zeta \geq  y_i - \beta_0 - \beta_1 (1+\zeta_i) x_i && \forall \zeta \in \R^n : \twonorm{\zeta} \leq \Omega \quad \forall i \in I \\
                    &  v_i + w_i \transp \zeta \geq  \beta_0 + \beta_1 (1+\zeta_i) x_i - y_i && \forall \zeta \in \R^n : \twonorm{\zeta} \leq \Omega \quad \forall i \in I,
\end{align*}
which after reformulation becomes:
\begin{align*}
\mbox{(AARC-R)} \qquad
&\min        \qquad  &&  d            && \\
&\mbox{s.t.} \qquad  &&  d \geq \Omega \twonorm{w} + \sum_{i=1}^n v_i                       && \\
&                    &&  v_i \geq  y_i - \beta_0 - \beta_1 x_i + \Omega \twonorm{\beta_1 x_i e_i + w_i} && \forall i \in I  \\
&                    &&  v_i \geq  \beta_0 + \beta_1 x_i - y_i + \Omega \twonorm{\beta_1 x_i e_i - w_i} && \forall i \in I,
\end{align*}
where $e_i$ is the $i^{th}$ unit vector.

We have run these four models on $1,000$ cases. In each case, we fixed the parameters at $\beta_0 = 2$, $\beta_1 = 5$ and $\sigma^2 = 1$. A case consists of 15 observations of $x_i$, generated uniformly in $[0, 100]$ i.i.d., $\zeta \in \R^n : \twonorm{\zeta} \leq \Omega$ uniformly and $\varepsilon_i \sim N(0,\sigma^2)$ i.i.d. Using these random draws, we have computed $y_i = \beta_0 + \beta_1 (1+\zeta_i) x_i + \varepsilon_i$. For the uncertainty region, we have picked $\Omega = 0.05$. For each of this cases, we solved the exact formulation \eqref{ex-reg-true}, AARC-R, QARC-R, RC-R, and the nominal problem using CVX and SDPT3. All solution times are very low, and hence, not mentioned.

Next to comparing the usual way by means of $v_{true}$, we also compare solutions by looking at how well $\beta_0$ and $\beta_1$ are estimated. In all $1,000$ cases, the AARC-R and QARC-R give the same estimates of $\beta_0$ and $\beta_1$. Hence, they are considered equal. Histograms, the mean, and $s^2$ statistics for $\beta_0$, $\beta_1$ and $v_{true}$ are shown in Figure \ref{fig:ex-regressie-hist}. All models do well in estimating the parameters $\beta_0$ and $\beta_1$, except the AARC-R and the QARC-R for $\beta_0$. 

The objective values (averaged over the $1,000$ cases) are listed in Table \ref{tbl:regression}. The AARC-R and QARC-R objective values are 22.7\% respectively 94.4\% closer to the optimum than the RC-R objective value. However, for the true value we have $v_{true}(x_{RC-R})=92.095$ while $v_{true}(x_{AARC-R})=v_{true}(x_{QARC-R})=99.322$. The RC-R outperforms the AARC-R and the QARC-R in 98\% of the cases, and even the nominal solution outperforms the RC-R in 99\% of the cases. So, even though the AARC-R and QARC-R provide a much better bound on the optimal true value, their true value is not necessarily better. This shows two things: it may be very misleading to compare robust solutions by their objective values, and using a better approximation of the true objective function might not improve the solution at all.

While we were able to solve the exact RC \eqref{ex-reg-true} directly, we have also looked at the efficiency of the cutting plane methods. This comparison is not based on the $1,000$ cases mentioned earlier, but on new cases where we varied the dimension of $\zeta$ (and consequently, the dimension of $x$). Each case was created in the same way as described before. Algorithms \ref{alg:itervertex} and \ref{alg:iterrlc} can run very quickly because the worst case $\zeta$ can be computed efficiently, as this is a simple case. It is therefore interesting to look at the number of iterations the algorithms take. We have generated 1,000 data sets in which we varied the number of observations. Algorithm \ref{alg:itervertex} adds only $3$ or $4$ constraints, independent of the number of observations. This shows that the method is not only effective for polyhedral uncertainty regions. The number of iterations of Algorithm \ref{alg:iterrlc} is shown in Figure \ref{fig:ex-regressie-numiter}, where it seems to be a square root function of the number of observations. The regression model $numiter_i = \alpha \sqrt{ numobs_i } + \varepsilon_i$ gives $\alpha=1.258$ with a standard error of $0.007$ and $R^2 = 0.966$. For 200 observations, the algorithm generates at most $18$ constraints while full EORLC would result in a model with $2^{200}$ constraints. The algorithm needs more iterations than the vertex enumeration algorithm, but is still effective.

\begin{table}
\caption{A comparison of different solution methods for the regression example averaged over $1,000$ datasets.\label{tbl:regression}}
{\begin{tabular}{lrr}
\toprule
Method                                 & $v_{method}$  & $v_{true}(x_{method})$ \\
\midrule
Nominal                                &  $36.326$ & $91.994$ \\
RC-R                                   & $224.096$ & $92.095$ \\
AARC-R                                 & $194.091$ & $99.322$ \\
QARC-R                                 &  $99.323$ & $99.322$ \\
Exact                                  &  $91.973$ & $91.973$ \\
\bottomrule
\end{tabular}}
\end{table}

\begin{figure}
	\centering
		\includegraphics[trim={3.2cm 1.5cm 3.1cm 1.7cm},clip,scale=0.6]{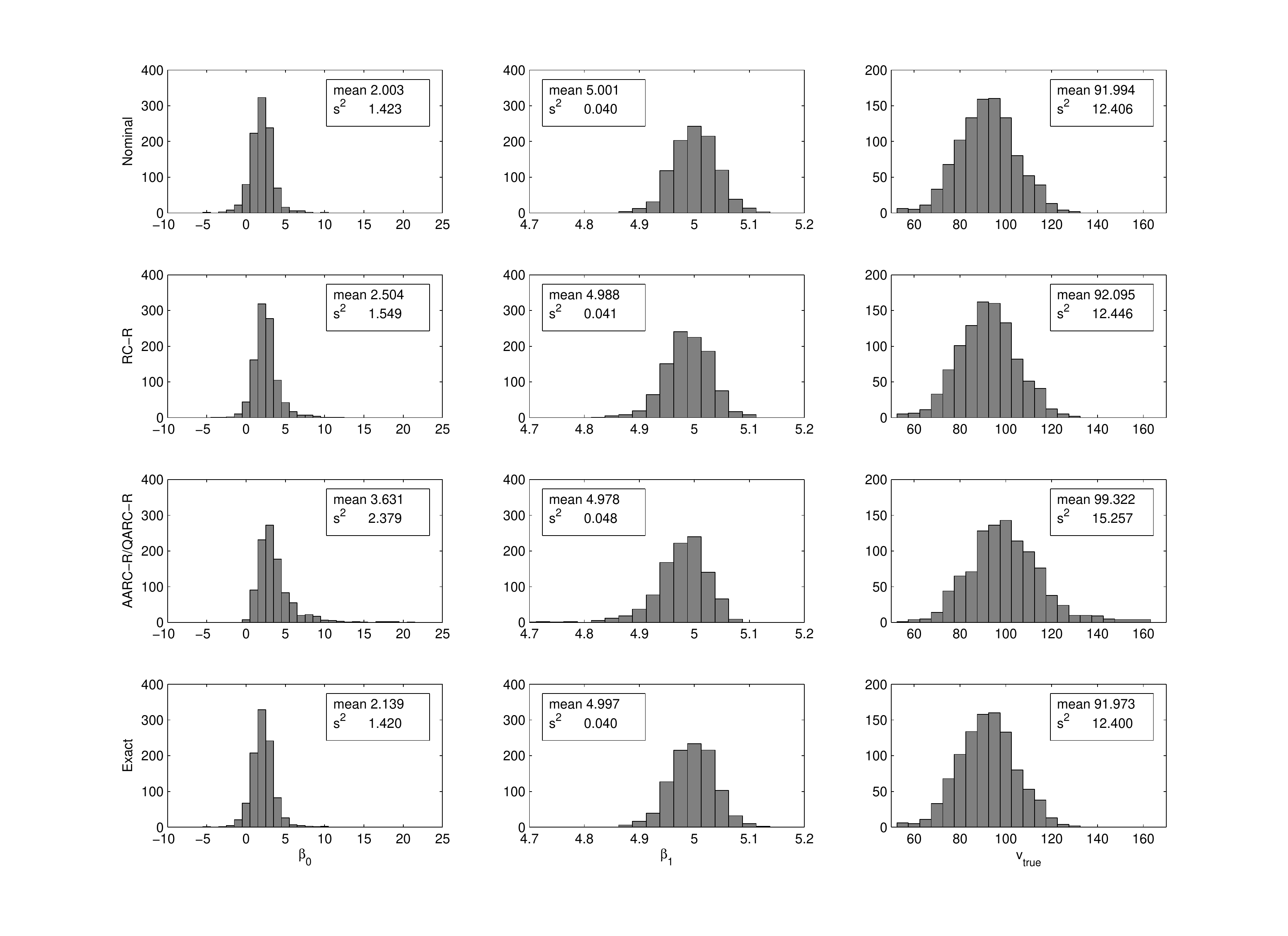}
	\caption{Comparison of the nominal solution, the RC-R solution, the AARC-R/QARC-R solution and the exact solution for the regression model of Section \ref{sec:exreg}. }
	\label{fig:ex-regressie-hist}
\end{figure}

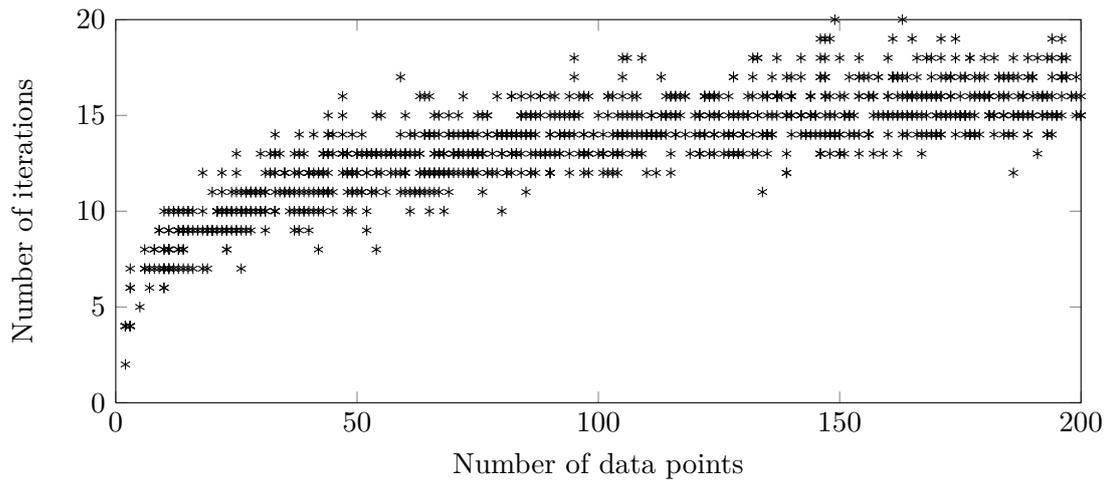
\begin{figure}
	\centering
		\input{regressie-numiter.tikz}
	\caption{The number of iterations needed by Algorithm \ref{alg:iterrlc} to find an optimal solution versus the number of observations in the regression model of Section \ref{sec:exreg}. }
	\label{fig:ex-regressie-numiter}
\end{figure}

\subsection{Brachytherapy}
High dose rate brachytherapy (HDR-BT) is a form of radiation therapy where a highly radioactive sealed source is inserted into a tumor for short time periods via approximately fifteen till twenty catheters. When the catheter positions are fixed, a treatment plan specifies for how long the radioactive source has to stay at which position inside the catheters. A perfect treatment plan delivers a prescribed dose to the tumor while not delivering any dose to the surrounding organs at risk. Because this is physically impossible, the goal is to find a good trade-off between irradiating the tumor and saving the organs at risk. The quality of a plan can be measured by means of calculation points. These are artificial points where the received dose can be computed and compared to a prescribed lower and upper bound, that are distributed inside and around the tumor and organs at risk. The dose in a specific calculation point $i$ can be computed as the sum of the individual contributions from every catheter $k$, which in turn is the sum of the dwell times of the individual dwell positions inside catheter $k$ multiplied by given dose rate vectors $d_{ik}$:

\[ \sum_{k \in K} d_{ik}\transp t_k. \]
If calculation point $i$ fails to satisfy the lower bound $L_i$ or the upper bound $U_i$, it contributes a linear penalty of $\alpha_i$ or $\beta_i$ (respectively) per unit of violation to the objective function. This results in the following optimization problem:

\[ \min_{t_k \in \R^n_+} \sum_{i \in I} \max \{0,	\alpha_i (L_i - \sum_{k \in K} d_{ik}\transp t_k), \beta_i (\sum_{k \in K} d_{ik}\transp t_k - U_i) \}. \]
This convex piecewise linear objective function is commonly used for treatment planning (\cite{Alterovitz2006,Karabis2009,Lessard2001}).

The parameters $d_{ik}$ are computed based on the catheter positions. These positions cannot be measured accurately, hence the data in the optimization problem is uncertain. We assume that the true catheter position is within some cone around the measured position, which is justified because one side of the catheter is fixed at a known position. We replace the cone with a polyhedral cone with a 10-sided base to end up with a polyhedral set, and we make the simplification that the vector $d_{ik}$ inside the cone is a convex combination of the vectors $d_{ik}$ at the sides of the cone. Thus, we only need to know the vectors $d_{ik}$ at the catheter positions corresponding to the $10$ edges of the cone connecting the apex to the base. Using these vectors as the columns of a matrix $B$, we can write $d_{ik} = B_{ik} \zeta_k$, where $\zeta_k \in \Delta^{|S|-1}$ ($\Delta^{|S|-1} = \{\zeta \in \R^{|S|} : e \transp \zeta = 1, \zeta \geq 0\}$, the standard simplex in $\R^{|S|}$), and $|S|$ is the number of sides of the base of the polyhedral cone. This gives the following RO problem:
\begin{align*}
\min  			\qquad	&v \\
\mbox{s.t.} \qquad	& \sum_i \max \{0, \alpha_i (L_i - \sum_{k \in K} (B_{ik} \zeta_k)\transp t_k), \\
                    & \hspace{3cm}  \beta_i (\sum_{k \in K} (B_{ik} \zeta_k)\transp t_k - U_i) \} \leq v && \forall \zeta_k \in \Delta^{|S|-1} (k \in K) \\
                    &t_{k}\geq 0                                          && \forall k\in K.
\end{align*}

The number of calculation points is usually in the order of magnitude of 5,000. We have data of one patient that has been exported from treatment planning software with a large number of calculation points. For the purpose of this paper, we have reduced the number of calculation points to $40$ by taking a random subset, because otherwise both the maximization step in Algorithm \ref{alg:iterrlc} and the AARC-R are intractable. The model does not lose its value from this reduction, because the objective could be a weighted average between the nominal objective based on all calculation point, and the robust objective based on a small fraction of calculation points that is well distributed inside the tumor.

The RC-R is based on the original constraint with the sum of $40$ maxima of $3$ functions. We have also solved the RC-R and AARC-R after splitting up the sum, as outlined in the last paragraph of Section \ref{sec:approx}. If we split the sum in groups of size 4, then the number of sums reduces to 10 while each term of sum is the maximum of $3^4$ functions.

The results for the different methods are listed in Table \ref{tbl:brachy}. The first observation is that the nominal solution has a true value that is five times the optimal value, so the nominal solution is nonrobust. The AARC-R has value $\approx 143$, which is much closer to the true value ($\approx 137$) than to the RC-R value ($\approx 229$), so the AARC-R almost closes the gap. Algorithm \ref{alg:iterrlc} is outperformed by Algorithm \ref{alg:itervertex}. Its long running time when the optimization problem for determining $v_{true}$ is not solved to optimality is not only due to the larger number of iterations, but more importantly, to the gradually increasing running time of solving $LP$ in each iteration. This step becomes so memory consuming, that the algorithm cannot finish with CPLEX on 32 bit hardware. Combining both algorithms reduces the number of iterations, as more work per iteration is done, but also increases the solution time in case the optimization problem for determining $v_{true}$ is not solved to optimality. Splitting up the sum does not give any benefit in terms of true value, though it gives a lower objective value when solving the RC-R.

\begin{table}
\caption{A comparison of different solution methods for the brachytherapy example.\label{tbl:brachy}}
{\begin{tabular}{lrrrr}
\toprule
Method                                   & Iterations & Sol. time (s) & $v_{method}$  & $v_{true}(x_{method})$ \\
\midrule
Nominal                                  &    $-$ &       $0$ &  $11.033$ & $711.362$ \\
RC-R                                     &    $-$ &       $1$ & $229.986$ & $157.277$ \\
AARC-R                                   &    $-$ &     $314$ & $143.104$ & $140.827$ \\[6pt]
RC-R (sum of 8 maxima of 243 functions)  &    $-$ &      $51$ & $186.219$ & $157.568$ \\
RC-R (sum of 10 maxima of 81 functions)  &    $-$ &      $17$ & $203.100$ & $157.030$ \\
RC-R (sum of 20 maxima 9 functions)      &    $-$ &       $2$ & $218.894$ & $158.438$ \\
AARC-R (sum of 10 maxima of 81 functions &    $-$ &  $136214$ & $141.481$ & $140.114$ \\
AARC-R (sum of 20 maxima 9 functions)    &    $-$ &    $2774$ & $142.447$ & $139.957$ \\[6pt]
Algorithm \ref{alg:itervertex}           &   $25$ &   $1,173$ & $137.628$ & $137.676$ \\
Algorithm \ref{alg:iterrlc}              &   $36$ &   $1,065$ & $137.629$ & $137.662$ \\
Combined                                 &   $20$ &     $894$ & $137.628$ & $137.650$ \\[6pt]
Algorithm \ref{alg:itervertex}$^\ast$    &   $53$ &     $234$ & $137.627$ & $137.668$ \\
Algorithm \ref{alg:iterrlc}$^\ast$       & $>400$ & $>50,000$ &       $-$ &       $-$ \\
Combined$^\ast$                          &   $41$ &     $353$ & $137.628$ & $137.676$ \\
\bottomrule
\end{tabular}}
{\scriptsize $\varepsilon=0.05$} \\
{\scriptsize $^\ast$ The optimization problem for determining $v_{true}$ is stopped as soon as the upper bound is at least $0.05$ higher than the lower bound }
\end{table}
\subsection{Inventory planning}
We consider a single item inventory model where backlogging is allowed to compare the AARC-R with the exact RC. At the beginning of each period, an order can be placed that is delivered instantly. At the end of each period, the holding and backlogging costs are $c_h$ and $c_b$ per unit, respectively. The objective is to minimize the costs:
\begin{align*}
\min_{q,w \in \R_+^{|T|}} \max_{d \in \Zset} \left(
		 \sum_{t=1}^{|T|} \max\{ c_h [x_0 + \sum_{i=1}^t q_i - d_i], c_b [x_0 + \sum_{i=1}^t q_i - d_i] \}
		 \right),
\end{align*}
where $x_0$ is the starting inventory, $q_i$ is the order quantity at time $i$, and $d_i$ is the uncertain demand at time $i$. In all formulations we allow $q_i$ to depend affinely on the demand in periods $1$ up to $i-1$.

If the uncertainty region is a box, the AARC-R turns out to be exact for the intervals we have tried. This is in accordance with the numbers reported by \cite{BenTal:2005:retailer}. We found that the AARC-R is no longer exact if the uncertainty region is an ellipsoid. Because demand is nonnegative, the ellipsoid is intersected with the nonnegative orthant:
\[ \Zset = \{ d \in \R^{|T|}_{+} : \twonorm{d-\bar{d}} \leq \Omega \}. \]
In our numerical study we have looked at 12 time periods, with parameters $\Omega = 10$, $\bar{d}=5$, $c_h=1$, and $c_b=2$.

Each determination of $v_{true}$ in Algorithm \ref{alg:iterrlc} takes up to $2$ minutes if continued to optimality using CPLEX, making it the most time consuming step in the algorithm. Because the dimension of the uncertain demand vector is $12$, a global solver might be faster. We have tried LGO 1.0, whose accuracy can be adjusted with the parameters ``maximal number of function evaluations" and ``maximal number of stalled evaluations", which initially both are 16,000. Every time the upper bound found by LGO is less than $0.1$ larger than the lower bound, the parameters are increased by 25\% until they exceed 1,000,000. This reflects the idea that it is still easy to find a violated constraint when the algorithm starts, but gradually becomes more difficult as the quality of the solution increases. Still, CPLEX finds better solutions in less time.

The RC-R is based on the original constraint with the sum of $12$ maxima of $2$ functions. We have also solved the RC-R and AARC-R after splitting up the sum, as outlined in the last paragraph of Section \ref{sec:approx}. In this example the uncertainty enters the model through time, so the most natural way of splitting up the summation is in consecutive time periods. We split the problem into two groups:
\begin{align*}
\min  			\qquad	&y_1 + y_2 \\
\mbox{s.t.} \qquad	&\sum_{t=1}^{6}  \max\{ c_h [x_0 + \sum_{i=1}^t q_i - d_i], c_b [x_0 + \sum_{i=1}^t q_i - d_i] \leq y_1  && \forall d \in \Zset \\
                    &\sum_{t=7}^{12} \max\{ c_h [x_0 + \sum_{i=1}^t q_i - d_i], c_b [x_0 + \sum_{i=1}^t q_i - d_i] \leq y_2  && \forall d \in \Zset,
\end{align*}
where again we allow $q_i$ to depend affinely on the demand in periods $1$ up to $i-1$. This reformulation introduces more constraints with the sum of maxima, but each sum contains less terms, hopefully resulting in a shorter solution time. The worst case $d$ may differ for the constraints that set $y_1$ and $y_2$, hence this reformulation is not exact.

\begin{table}
\caption{A comparison of different solution methods for the inventory example.\label{tbl:inventory}}
{\begin{tabular}{lrrrr}
\toprule
Method                                              & Iterations & Sol. time (min) & $v_{method}$  & $v_{true}(x_{method})$ \\
\midrule
Nominal                                             &   $-$ &   $0$ &   $0.000$ & $509.903$ \\
RC-R                                                &   $-$ &   $0$ & $120.000$ &  $94.641$ \\
AARC-R                                              &   $-$ &   $0$ & $120.000$ &  $93.315$ \\[6pt]
EORLC                                               &   $-$ &  $10$ &  $48.750$ &  $48.750$ \\
RC-R (sum of 2 maxima of 64 functions)              &   $-$ &   $0$ &  $68.613$ &  $54.162$ \\
RC-R (sum of 3 maxima of 16 functions)              &   $-$ &   $0$ &  $83.631$ &  $58.140$ \\
RC-R (sum of 4 maxima of  8 functions)              &   $-$ &   $0$ &  $94.456$ &  $66.637$ \\
RC-R (sum of 6 maxima of  4 functions)              &   $-$ &   $0$ & $107.627$ &  $77.957$ \\ [6pt]
Algorithm \ref{alg:itervertex}                      & $123$ &  $76$ &  $48.749$ &  $48.796$ \\
Algorithm \ref{alg:iterrlc}                         &  $77$ &  $36$ &  $48.752$ &  $48.802$ \\
Combined                                            &  $79$ &  $45$ &  $48.755$ &  $48.798$ \\ [6pt]
Algorithm \ref{alg:itervertex}$^\ast$               & $309$ & $180$ &  $48.752$ &  $48.827$ \\
Algorithm \ref{alg:iterrlc}$^\ast$                  & $112$ &   $9$ &  $48.753$ &  $48.790$ \\
Combined $^\ast$                                    & $103$ &  $29$ &  $48.755$ &  $48.825$ \\
\bottomrule
\end{tabular}}
{\scriptsize $\varepsilon=0.1$} \\
{\scriptsize $^\ast$ The optimization problem for determining $v_{true}$ is stopped as soon as the upper bound is at least $0.1$ higher than the lower bound }
\end{table}

The results are listed in Table \ref{tbl:inventory}. Again note that the order policies ($q_i$) are adjustable in all formulations, including the RC-R, so the differences in this table are caused only by different reformulations of the sum of maxima. The first observation is that the AARC-R gives a very small gain over the RC-R, but still has almost twice the optimal value. So, making the analysis variables adjustable does not significantly improve the solution. Algorithm \ref{alg:iterrlc} is the fastest cutting plane method, requiring approximately the same number of iterations when combined with Algorithm \ref{alg:itervertex}. The sum splitting method significantly reduces the computation time at the cost of nonoptimality. It performs much better than the AARC-R, both in optimal value and in true value. Using the AARC-R on the splitted sums gives a very small gain over the RC-R on the splitted sums, just as for the full problem, so it is not listed in the table. The large true value of the nominal solution comes from the fact that the order sizes are fixed in advance and are not adjusted to observed demand.

\section{Conclusions}
\label{sec:conclusions}
Because RO is applied constraint-wise, it is very important how constraints are formulated. In this paper we list several approaches to an inequality constraint containing the sum of maxima of linear functions. The RC-R, often used in the literature, is the most pessimistic approach. It is obtained by first reformulating the deterministic constraint into linear constraints using analysis variables, and then applying RO. Its pessimism can be reduced by replacing the analysis variables with linear decision rules before applying RO, which gives the AARC-R. The AARC-R seems to work well for the practical problems we analyzed with polyhedral uncertainty regions, but we have constructed an example with polyhedral uncertainty where the value of the AARC-R is 100\% higher than the true value. Nonlinear decision rules may give better results, but are computationally more challenging. The conservatism of the approximations can be reduced by combining several $\max$ expressions before reformulating. Especially for ellipsoidal uncertainty this method gives much better solutions at the cost of a slightly higher solution time.

In many cases it is not necessary to use an approximation because an exact reformulation can be practically solved. We identify four special cases in which an exact reformulation is often tractable. For the general case we give two exact general methods: vertex enumeration and EORLC. Both methods may result in very large optimization models, but cutting plane methods can be used to handle this. Vertex enumeration adds a set of constraints for every vertex of the uncertainty region, so this method is preferred if the uncertainty region has a low number of vertices. Surprisingly, its cutting plane version is also capable of solving problems where the uncertainty region has an infinite number of extreme points efficiently. EORLC is preferred if the number of terms with a maximum function is low. If it is not clear in advance which cutting plane method is faster, both methods have to be tried because our numerical examples do not show a clear preference. Both methods can be combined, but we have not found a situation in which it is beneficial to do so. EORLC can be combined with approximations, which is a new idea that can be used for large scale problems which shows promising numerical results.

The RC-R is often used in the literature while less conservative approaches could have been applied, mostly without explicitly mentioning that their approach is conservative. In the paper by \cite{Kropat2010}, the exact method EORLC would have increased the number of constraints by only a factor four while reducing the number of variables with almost a factor 2 and not changing the structure of the problem. The same authors applied vertex enumeration to an ellipsoidal constraint with polyhedral uncertainty \citep{Ozmen2011}, which gives an exact reformulation. \cite{BertsimasThiele2006}, \cite{Wei2009} and \cite{JoseAlem2012139} apply the RC-R to a problem with polyhedral uncertainty. Our results, and also the numbers reported by \cite{BenTal:2005:retailer}, show that the AARC-R often gives less conservative solutions while it has the advantage that the problem remains linear. For inventory problems in general, when the order quantities are made adjustable, then oftenly also the analysis variables are made adjustable. We show that the latter is not beneficial for ellipsoidal uncertainty, and that both the RC-R and the AARC-R give very bad solutions. For small planning horizons, an exact method has to be used, while for larger horizons splitting the sum in small groups and applying EORLC on the groups significantly improves the solution. \cite{Ng2010557} solve a lot allocation problem with ellipsoidal uncertainty. Because the problem is computationally challenging, they solve a problem equivalent to our RC-R using Benders decomposition. Even though their problem is so challenging that even the simple RC-R cannot be solved within a day and their speed-ups are necessary to solve the problem efficiently, it is still interesting to know the conservatism of their approach. We have been able to get a suboptimal solution with cutting planes based on vertex enumeration, where both the minimization and the maximization step were stopped before optimality. The solution we got after four hours has a true value of 26.8, whereas the RC-R (which we tried to solve as an MILP) has a value between 28.7 and 47.2. So the objective value of the solution proposed by \cite{Ng2010557} is at least 7-76\% too pessimistic.

From our numerical examples it becomes clear that the RC-R is not necessarily better than the nominal problem. Neither of the two optimizes the true problem, so it cannot be determined a priori which one has a better true value. The same holds for the AARC-R and the RC-R: If the AARC-R gives a much lower value then at least it provides a guarantee on the worst case, but the RC-R may still outperform the AARC-R. When using an approximation, it is therefore crucial to measure its quality. This can be accomplished by comparing the true value of the solution of the approximation with the value of an exact formulation. If the problem is too large to be solved exactly, the comparison may be based on a smaller instance with similar structure.

\section*{Acknowledgments}
We would like to thank A. Ben-Tal from Technion (Haifa) and R. Sotirov from Tilburg University (The Netherlands) for their input. We thank M. Balvert from Tilburg University for the idea behind the reformulation in Appendix \ref{sec:deriv-marleen}. Moreover, we like to thank the reviewers for their comments that have significantly improved the quality of the paper.
\appendix
\section{Derivation of AARC-R using Fenchel's duality}
\label{sec:rewrite-fenchel}

In this appendix we apply Fenchel's duality to robust constraints, a technique introduced in RO by \cite{BenTal2011}. First we will briefly mention the general theory, then we will apply it to constraints of general form, and finally we will apply the general results to constraint \eqref{cons:general} and show that the result is the same as the AARC-R.

\addtocontents{toc}{\protect\setcounter{tocdepth}{1}}
\subsection{Fenchel's duality theorem}
\addtocontents{toc}{\protect\setcounter{tocdepth}{2}}
We start with some definitions that are necessary to formulate Fenchel's theorem:
\begin{definition}
A function $\phi$ is proper convex if it is convex, its codomain is $\R \cup \{\infty\}$, and $\phi(x) < \infty$ for at least one $x$.
\end{definition}
\begin{definition}
A function $\psi$ is proper concave if it is concave, its codomain is $\R \cup \{-\infty\}$, and $\psi(x) > -\infty$ for at least one $x$.
\end{definition}
\begin{theorem} (Fenchel's duality \cite[p.~327]{Rockafellar})
Let $\phi$ be a proper convex function on $\R^n$, let $\psi$ be a proper concave function on $\R^n$, and let either of the following conditions be satisifed:
\begin{itemize}
\item ri(dom $\phi$) $\cap$ ri(dom $\psi$) $\neq \emptyset$
\item $\phi$ and $\psi$ are closed, and ri(dom $\phi^*$) $\cap$ ri(dom $\psi_*$) $\neq \emptyset$,
\end{itemize}
where ri is the relative interior, dom is the effective domain (dom $\phi = \{x : \phi(x) < \infty\}$), and $\phi^*$ and $\psi_*$ are the convex and concave conjugate of $\phi$ and $\psi$, respectively. That is,
\[ \phi^*(s) = \sup_x \{s \transp x - \phi(x)\} \]
\[ \psi_*(s) = \inf_x \{s \transp x - \psi(x)\}. \]
Then the following equality holds
\[ \inf_x  \{ \phi(x) - \psi(x) \}   =   \sup_s \{ \psi_*(s) - \phi^*(s) \}. \]
\end{theorem}

\subsubsection{Fenchel's duality applied to a robust constraint of general form.}

We focus on the general robust constraint
\begin{equation}\label{fench-prestart} g(\zeta,x) \leq d \qquad \forall \zeta \in \Zset, \end{equation}
where $g$ is a proper concave function of $\zeta$ for any fixed value of $x$, and the condition for Fenchel's duality is satisfied with respect to the first argument for any fixed value of $x$. Because values of $g$ are not of interest when $\zeta \notin \Zset$, we assume that $g(\zeta,x) = -\infty$ for all $\zeta \notin \Zset$. We also assume that $\Zset$ is a compact set so that this constraint is equivalent to:
\begin{equation}\label{fench-start} \max_{\zeta \in \R^L}   \{ g(\zeta,x) - \delta_\Zset(\zeta)   \} \leq d, \end{equation}
with $\delta_\Zset$ the indicator function ($\delta_\Zset(\zeta) = 0$ if $\zeta \in \Zset,$ and $\infty$ otherwise). We can rewrite the left-hand side by applying Fenchel's duality:
\begin{equation*} \min_{s \in \R^L}  \{  \delta_\Zset^*(s) - g_*(s,x)   \} \leq d, \end{equation*}
which holds if and only if there exists some $s \in \R^L$ such that:
\begin{equation}\label{fench-simplified} \delta_\Zset^*(s) - g_*(s,x) \leq d. \end{equation}
Note that this constraint is convex in $s$. Because every step is `if and only if', constraint \eqref{fench-prestart} is equivalent to constraint \eqref{fench-simplified}.

\subsubsection{$g$ is the sum of other functions.}

If $g$ can be written as the sum of several other functions, it might be impossible or very difficult to find a closed form solution for its conjugate function. Suppose we have a constraint of the form:
\begin{equation}\label{fench-sum1} \sum_{i \in I} g_i(\zeta,x) \leq d  \qquad \forall \zeta \in \Zset, \end{equation}
which is constraint \eqref{fench-prestart} with $g = \sum_{i \in I} g_i$. If we want to formulate an equivalent constraint using Fenchel's duality, we need the concave conjugate of $g$. Under some mild assumptions on $g_i$, it turns out to be sufficient to have closed form solutions for the conjugates of $g_i$. The following lemma appears to be very useful \cite[p.~145]{Rockafellar}:
\begin{lemma} \label{conj-sum-concave}
Let $\psi_i$ ($i \in I$) be proper concave functions on $\R^n$. If $\cap_{i \in I}$ ri(dom $\psi_i$) $\neq \emptyset$ then
\[ (\sum_{i \in I} \psi_i)_*(s) = \sup_{\sum_{i \in I}s_i=s} \{ \sum_{i \in I} (\psi_i)_*(s_i) \}, \]
and the supremum is attained.
\end{lemma}
Applying this lemma, we can rewrite constraint \eqref{fench-sum1} as:
\begin{align*}
 \delta_\Zset^*(s) - \max_{\sum_{i \in I}s_i = s} \{ \sum_{i \in I} (g_i)_*(s_i,x) \} \leq d,
\end{align*}
which is valid if and only if there exists $s_i \in \R^L (i \in I)$ such that
\begin{align}
\label{fench-sum3}
  \delta_\Zset^*(\sum_{i \in I} s_i) - \sum_{i \in I}  (g_i)_*(s_i,x)  \leq d.
\end{align}
So, if all conjugates have closed form solutions, the reformulated constraint also has a closed form. If all conjugates do not have a closed form, this reformulation is still useful because it allows computing each term separately. We will show this for piecewise linear convex functions later in this section.
It should again be noted that this constraint is convex in $s_i$.

\subsubsection{$f$ is not convex or $g$ is not concave.}

Let us investigate the consequences to the reformulation if $g$ is not concave. We assume $g$ is finite on some non-empty set $\Zset$, and $\infty$ elsewhere, so that its conjugate is non-trivial. Fenchel's inequality \cite[p.~105]{Rockafellar} states that:

\[ \delta_\Zset(\zeta) + \delta_\Zset^*(s)    = \delta_\Zset(\zeta) + \sup_{\zeta' \in \Zset} \{ {\zeta'}\transp s - \delta_\Zset(\zeta') \}     \geq \delta_\Zset(\zeta) + \zeta \transp s - \delta_\Zset(\zeta)     = \zeta \transp s   \]
\[ g(\zeta,x) + g_*(s,x)    = g(\zeta,x) + \inf_{\zeta' \in \Zset} \{ {\zeta'}\transp s - g(\zeta',x) \}     \leq g(\zeta,x) + \zeta \transp s - g(\zeta,x)     = \zeta \transp s,   \]
hence:

\[ \delta_\Zset(\zeta) + \delta_\Zset^*(s)    \geq    g(\zeta,x) + g_*(s,x),     \]
and consequently:

\[ g(\zeta,x) - \delta_\Zset(\zeta)    \leq    \delta_\Zset^*(s) - g_*(s,x).   \]
This implies that if constraint \eqref{fench-simplified} is satisfied, then so is constraint \eqref{fench-start}, but the reverse implication is not necessarily true. Hence, constraint \eqref{fench-simplified} is a conservative reformulation of constraint \eqref{fench-start}.

Also, Lemma \ref{conj-sum-concave} no longer holds with equality. We can rewrite it as an inequality:
\begin{align*}
  (\sum_{i \in I} \psi_i)_*(s) &= \inf_{t} \{ s \transp t - \sum_{i \in I} \psi_i(t) \} \\
                               &= \inf_{t} \{ \sum_{i \in I} s_i \transp t - \psi_i(t) \},
\end{align*}
for any $s_1,...,s_{m} \in \R^L$ for which $\sum_{i \in I} s_i = s$. So in particular:
\begin{align*}
  (\sum_{i \in I} \psi_i)_*(s) &=    \sup_{\sum_{i \in I}s_i=s} \{ \inf_{t} \{ \sum_{i \in I} s_i \transp t - \psi_i(t) \} \} \\
                               &\geq \sup_{\sum_{i \in I}s_i=s} \{ \sum_{i \in I} \inf_{t} \{ s_i \transp t - \psi_i(t) \} \} \\
                               &=    \sup_{\sum_{i \in I}s_i=s} \{ \sum_{i \in I} (\psi_i)_*(s_i) \}.
\end{align*}
This implies that if constraint \eqref{fench-sum3} is satisfied, then so is constraint \eqref{fench-sum1}, but the reverse implication is not necessarily true. Hence, constraint \eqref{fench-sum3} can be seen as a conservative reformulation of constraint \eqref{fench-sum1}.

\subsubsection{$g$ is the sum of pointwise maxima of linear functions.}
Let us first derive the conjugates of some functions before we arrive at the theorem:
\begin{align*}
    \delta_\Zset^*(s) &= \sup_{\zeta \in \R^L} \{s \transp \zeta - \delta_\Zset(\zeta)\} = \sup_{\zeta \in \Zset}\{s \transp \zeta\} = \max_{\zeta \in \Zset}\{s \transp \zeta\},
\end{align*}
and
\begin{align*}
   (\max_{j \in J} \{ \ell_{ij}(\zeta,x) \})_*(s_i,x) &= \inf_{\zeta \in \Zset} \{  s_i \transp \zeta - \max_{j \in J} \{ \ell_{ij}(\zeta,x) \} \}  \\
                                                      &= \inf_{\zeta \in \Zset} \{ \min_{j \in J}\{ s_i \transp \zeta - \ell_{ij}(\zeta,x) \} \}    \\
                                                      &= \min_{j \in J}\{ \inf_{\zeta \in \Zset} \{ s_i \transp \zeta - \ell_{ij}(\zeta,x) \} \}.
\end{align*}
\begin{theorem}
Applying Fenchel's duality to:
\begin{align} \max_{\zeta \in \Zset} \sum_{i \in I} \max_{j \in J} \{ \ell_{ij}(\zeta,x) \} \leq d, \label{cons:eqtoaarcr} \end{align}
gives a formulation that is equivalent to the AARC-R.
\end{theorem}
\begin{proof}
Constraint \eqref{cons:eqtoaarcr} is equivalent to constraint \eqref{fench-sum1} with $g_i(\zeta,x) = \max_{j \in J} \{ \ell_{ij}(\zeta,x) \}$ for $\zeta$ in $\Zset$ and $g_i(\zeta,x) = -\infty$ otherwise. For for any fixed $x$, $g_i$ is not concave in $\zeta$ so we will end up with a conservative instead of an equivalent reformulation. If we fill in the conjugate functions in constraint \eqref{fench-sum3}, the following conservative reformulation is obtained:
\begin{equation*} \max_{\zeta \in \Zset}\{\sum_{i \in I} s_i \transp \zeta\} - \sum_{i \in I} \min_{j \in J}\{ \inf_{\zeta \in \Zset} \{ s_i \transp \zeta - \ell_{ij}(\zeta,x) \} \} \leq d. \end{equation*}
If we model the second terms as $\sum_{i \in I} z_i$, we can write this as:
\begin{align*}
	\sum_{i \in I} s_i \transp \zeta - \sum_{i \in I} z_i &\leq d && \forall \zeta \in \Zset \\
	z_i            &\leq s_i \transp \zeta - \ell_{ij}(\zeta,x)   && \forall \zeta \in \Zset \quad \forall i \in I \quad \forall j \in J,
\end{align*}
and by rearranging the terms in each constraint we obtain:
\begin{align*}
	\sum_{i \in I} [(-z_i) + s_i \transp \zeta] &\leq d                   &&\forall \zeta \in \Zset \\
	(-z_i) + s_i \transp \zeta                  &\geq \ell_{ij}(\zeta,x)  &&\forall \zeta \in \Zset \quad \forall i \in I \quad \forall j \in J,
\end{align*}
which is the same as the AARC-R. \qed
\end{proof}
\section{Derivation of AARC-R by reformulating the nonrobust constraint}
\label{sec:deriv-marleen}
In this appendix we give a different derivation of the AARC-R of constraint \eqref{cons:general} when both the biaffine functions and the uncertainty region are separable in the following way:
\begin{align} \label{eq:mrl1}
   \sum_{i \in I} \max_{j \in J} \{ \sum_{k \in K} \ell_{ijk}(\zeta_k,x) \} \leq d \qquad \forall \zeta_k \in \Zset_k \quad (k \in K),
\end{align}
and $\Zset_k$ is the convex hull of different scenarios $\zeta_k^s$ ($s \in S$). An example where this constraint is commonly used, is HDR brachytherapy optimization (\cite{Alterovitz2006,Karabis2009,Lessard2001}). If the summation over $k$ were outside the $max$ expression, then an analysis variable could be used for every $k$ without introducing any conservatism. Vertex enumeration can then be done on every $\Zset_k$ separately. For problems not affected by uncertainty, we show that indeed there is an equivalent formulation where the summation over $k$ is outside the $max$ expression. Then we show equivalence to the AARC-R if there is uncertainty. First, we prove the following equality for fixed $x$ and $\zeta_k$:
\begin{lemma}\label{lem:lem1}
\begin{align}
 \sum_{i \in I} \max_{j \in J} \{ \sum_{k \in K} \ell_{ijk}(\zeta_k,x) \}
 =
 \min_{y \in \R^{|I||J||K|} : \sum_{k \in K} y_{ijk} = 0} \sum_{k \in K} \sum_{i \in I} \max_{j \in J} \{ y_{ijk} + \ell_{ijk}(\zeta_k,x) \}. \label{eq:mrl2}
\end{align}
\end{lemma}
\begin{proof}
Note that:
\begin{align*}
 \sum_{k \in K} \sum_{i \in I} \max_{j \in J} \{ y_{ijk} + \ell_{ijk}(\zeta_k,x) \}
 \geq
 \sum_{k \in K} \sum_{i \in I} y_{ij(i)k} + \ell_{ij(i)k}(\zeta_k,x)
 = \sum_{i \in I} \sum_{k \in K} \ell_{ij(i)k}(\zeta_k,x) \quad \forall j(i) \in J, \\
\end{align*}
for any $j(i)$ in $J$, so in particular:
\begin{align*}
 \sum_{k \in K} \sum_{i \in I} \max_{j \in J} \{ y_{ijk} + \ell_{ijk}(\zeta_k,x) \} \geq \sum_{i \in I} \max_{j \in J} \{ \sum_{k \in K} \ell_{ijk}(\zeta_k,x) \},
\end{align*}
for any $y$, so in particular for the minimum. Hence, the right hand side of \eqref{eq:mrl2} is at least as large as the left hand side. On the other hand, given a feasible point for the left hand side of \eqref{eq:mrl2}, we can always construct a feasible point for the right hand side with equal value by taking the same $x$, and $y_{ijk} = \frac{1}{|K|} \sum_{k' \in K} \ell_{ijk'}(\zeta_{k'},x) - \ell_{ijk}(\zeta_k,x)$:
\begin{align*}
       \sum_{k \in K} \sum_{i \in I} \max_{j \in J} \{ y_{ijk} + \ell_{ijk}(\zeta_k,x) \}
  =    \sum_{i \in I} \sum_{k \in K} \max_{j \in J} \{ \frac{1}{|K|} \sum_{k' \in K} \ell_{ijk'}(\zeta_{k'},x) \}
  =    \sum_{i \in I} \max_{j \in J} \{ \sum_{k' \in K} \ell_{ijk'}(\zeta_{k'},x) \}.
\end{align*} \qed
\end{proof}
\begin{lemma}\label{lem:lem2}
A conservative reformulation of constraint \eqref{eq:mrl1} is given by:
\begin{align*}
&   \sum_{k \in K} \sum_{i \in I} \max_{j \in J} \{ y_{ijk} + \ell_{ijk}(\zeta_k,x) \} \leq d && \forall \zeta_k \in \Zset_k \quad (k \in K) \\
&   \sum_{k \in K} y_{ijk} = 0                                                                && \forall i \in I \quad \forall j \in J.
\end{align*}
\end{lemma}
\begin{proof}
We use the proof of Lemma \ref{lem:lem1}. The first part of the proof still holds, so the equality in constraint \eqref{eq:mrl2} becomes a ``$\leq$''. The construction of the feasible point in the second part of the proof depends on the value $\ell_{ijk}(\zeta_k,x)$. This value does not exist in the robust constraint \eqref{eq:mrl1}, since there is no single $\zeta_k$. Note that the left hand sides of \eqref{eq:mrl1} and \eqref{eq:mrl2} are the same, so by replacing the left hand side of \eqref{eq:mrl1} with the right hand side of \eqref{eq:mrl2}, the conservative reformulation is obtained. \qed
\end{proof}
The uncertainty is separable per $k$ in the first constraint of the conservative reformulation. Hence, analysis variables per $k$ can replace each term without changing the solution. Vertex enumeration can then be done for every $\Zset_k$ separately:
\begin{align}
& \sum_{k \in K} z_k \leq d \label{marlve} \\
& z_k \geq \sum_{i \in I} w_{iks}                   && \forall s \in S \notag \\
& w_{iks} \geq y_{ijk} + \ell_{ijk}(\zeta_k^s,x)    && \forall i \in I \quad \forall s \in S \quad \forall k \in K \notag \\
& \sum_{k \in K} y_{ijk} = 0                        && \forall i \in I \quad \forall j \in J. \notag
\end{align}
It remains to show that this formulation is equivalent to the AARC-R of constraint \eqref{eq:mrl1}.
\begin{theorem}
The conservative reformulation in Lemma \ref{lem:lem2} is the AARC-R in case of scenario generated uncertainty.
\end{theorem}
\begin{proof}
The AARC-R of constraint \eqref{eq:mrl1} is given by:
\begin{align}
\sum_{i \in I} \left( v_i + \sum_{k \in K} w_{ik} \transp \zeta_k \right) &\leq d && \forall \zeta_k \in \Zset_k \quad (k \in K) \notag \\
v_i + \sum_{k \in K} w_{ik} \transp \zeta_k &\geq \sum_{k \in K} \ell_{ij}(\zeta_k,x) && \forall i \in I \quad \forall j \in J \quad \forall \zeta_k \in \Zset_k \quad (k \in K). \label{cons:mrl4}
\end{align}
Constraint \eqref{cons:mrl4} can be reformulated as:
\begin{align*}
&w_{ik} \transp \zeta_k \geq y_{ijk} + \ell_{ij}(\zeta_k,x) && \forall i \in I \quad \forall j \in J \quad \forall \zeta_k \in \Zset_k \quad \forall k \in K \\
&\sum_{k \in K} y_{ijk} = -v_i  && \forall i \in I \quad \forall j \in J.
\end{align*}
We have assumed scenario generated uncertainty, so our uncertainty region is $\Zset_k = \Delta^{|S|-1}$, the standard simplex in $\R^{|S|}$. Suppose an optimal solution has $v_i \neq 0$, then an optimal solution with $v_i=0$ can be obtained by increasing all elements of $w_{ik}$ for a single random $k$ with $v_i$ because the elements of $\zeta_k$ sum to $1$. Hence we fix $v_i=0$. The AARC-R can then be formulated as:
\begin{align*}
 \sum_{k \in K} z_k &\leq d \\
 z_k &\geq \sum_{i \in I} w_{ik} \transp \zeta_k                   && \forall \zeta_k \in \Zset_k \quad (k \in K) \\
 w_{ik} \transp \zeta_k &\geq y_{ijk} + \ell_{ij}(\zeta_k,x)       && \forall i \in I \quad \forall j \in J \quad \forall \zeta_k \in \Zset_k \quad \forall k \in K \\
\sum_{k \in K} y_{ijk} &= 0                                        && \forall i \in I \quad \forall j \in J. \notag
\end{align*}
Let $w_{iks}$ denote the $s^{th}$ component of $w_{ik}$. Equivalence to \eqref{marlve} now follows from vertex enumeration. \qed
\end{proof}
\section{Derivation of the QARC-R for an ellipsoidal uncertainty region}
\label{sec:deriv-qarc-ellipsoidal}
In this appendix the SDP reformulation of the QARC-R of \eqref{cons:general} is derived for an ellipsoidal uncertainty region. For simplicity, we assume $\ell_{ij}(\zeta,x)$ and  $\ell$ are bilinear functions in the parameters. Therefore, they can be expressed as $\ell(\zeta,x) = \zeta \transp Lx$ and $\ell_{ij}(\zeta,x) = \zeta \transp L_{ij}x$, respectively, for some matrices $L$ and $L_{ij}$. The QARC-R is given by:
\begin{align*}
\textrm{(QARC-R)} \qquad \zeta \transp L x + \sum_{i \in I} \left( v_i + w_i \transp \zeta + \zeta \transp W_i \zeta \right) &\leq d && \forall \zeta \in \R^L : \twonorm{\zeta} \leq \Omega \\
 v_i + w_i \transp \zeta + \zeta \transp W_i \zeta  &\geq \zeta \transp L_{ij} x && \forall \zeta \in \R^L : \twonorm{\zeta} \leq \Omega \quad \forall i \in I \quad \forall j \in J.
\end{align*}
By application of the $\mathcal{S}$--lemma \cite[Lemma 6.5.3]{BenTal}, this can be reformulated as the following LMIs:
\begin{align*}
  \sum_{i \in I}
  \begin{pmatrix}
   W_i                     & \frac{1}{2} w_i \\
  \frac{1}{2}  w_i & v_i             \\
  \end{pmatrix}
  +
  \begin{pmatrix}
  0             & \frac{1}{2}Lx \\
  \frac{1}{2}Lx & -d \\
  \end{pmatrix}
  &\preceq
  \lambda_0
  \begin{pmatrix}
  I & 0 \\
  0 & -\Omega^2 \\
  \end{pmatrix}\\
- \begin{pmatrix}
  W_i                                      & \frac{1}{2}\left( w_i - L_{ij} x \right) \\
  \frac{1}{2}\left( w_i - L_{ij} x \right) & v_i                                      \\
  \end{pmatrix}
  &\preceq
  \lambda_{ij}
  \begin{pmatrix}
  I & 0 \\
  0 & -\Omega^2 \\
  \end{pmatrix}  \qquad \forall i \in I \quad \forall j \in J \\
  \lambda_0 \geq 0, \lambda_{ij} \geq 0.
\end{align*}
The QARC-R has $|I||J|+1$ LMIs of size $|L+1|$.

\bibliographystyle{abbrvnatnew}
\bibliography{library}

\end{document}

%% file: regressie-numiter.tikz
\begin{tikzpicture}

\begin{axis}[%
width=5in,
height=2in,
scale only axis,
xmin=0,
xmax=200,
xtick = {0,50,100,150,200},
xlabel={Number of data points},
ymin=0,
ymax=20,
ylabel={Number of iterations}
]
\addplot [
color=black,
only marks,
mark=asterisk,
mark options={solid},
forget plot
]
table[row sep=crcr]{
73 13\\
154 17\\
192 16\\
26 10\\
43 11\\
63 12\\
13 8\\
147 15\\
128 15\\
38 13\\
28 9\\
33 13\\
36 11\\
60 13\\
30 10\\
132 17\\
137 16\\
60 12\\
93 15\\
62 11\\
192 16\\
107 15\\
101 12\\
28 11\\
136 15\\
82 14\\
148 13\\
169 18\\
63 12\\
30 13\\
60 13\\
74 13\\
116 16\\
45 14\\
146 13\\
147 17\\
50 12\\
174 15\\
135 13\\
102 16\\
113 14\\
171 16\\
190 17\\
95 15\\
59 11\\
36 10\\
54 12\\
82 13\\
35 12\\
137 16\\
27 11\\
170 14\\
15 10\\
100 13\\
51 14\\
101 13\\
13 9\\
47 13\\
9 9\\
92 13\\
20 11\\
160 18\\
8 8\\
35 12\\
52 9\\
188 16\\
47 13\\
75 15\\
179 14\\
84 13\\
163 14\\
196 18\\
71 13\\
38 12\\
12 10\\
69 15\\
140 16\\
94 16\\
161 17\\
196 17\\
121 13\\
86 13\\
181 15\\
122 16\\
18 7\\
139 12\\
178 14\\
113 14\\
177 16\\
167 13\\
72 16\\
88 16\\
173 17\\
53 13\\
172 17\\
25 11\\
15 10\\
83 14\\
97 12\\
153 14\\
154 13\\
67 12\\
167 14\\
20 9\\
41 12\\
149 15\\
186 17\\
33 11\\
151 15\\
52 12\\
134 14\\
132 14\\
54 12\\
132 16\\
193 14\\
49 14\\
191 16\\
126 14\\
184 17\\
22 9\\
26 10\\
173 16\\
74 12\\
169 17\\
185 17\\
66 11\\
47 13\\
175 16\\
45 10\\
11 9\\
178 16\\
147 16\\
116 15\\
86 16\\
129 13\\
154 17\\
82 16\\
107 14\\
18 9\\
191 13\\
106 14\\
88 16\\
47 13\\
140 15\\
156 14\\
146 14\\
196 17\\
189 15\\
135 15\\
55 12\\
128 15\\
156 14\\
12 7\\
164 16\\
84 14\\
128 17\\
171 16\\
130 14\\
160 16\\
48 12\\
198 16\\
74 12\\
148 19\\
174 14\\
104 15\\
86 14\\
113 15\\
105 13\\
153 15\\
7 6\\
199 17\\
95 17\\
90 15\\
192 16\\
123 15\\
79 12\\
176 17\\
186 12\\
68 12\\
150 18\\
127 13\\
117 16\\
95 18\\
67 11\\
20 9\\
130 14\\
156 16\\
82 13\\
63 11\\
191 18\\
124 14\\
168 16\\
176 15\\
112 14\\
164 14\\
57 13\\
14 9\\
49 10\\
72 12\\
57 13\\
121 15\\
127 14\\
184 15\\
130 15\\
33 10\\
77 14\\
193 16\\
146 18\\
31 10\\
39 10\\
77 12\\
43 13\\
21 9\\
69 14\\
81 14\\
186 18\\
182 15\\
194 17\\
104 14\\
200 15\\
178 15\\
161 15\\
199 16\\
43 12\\
23 9\\
43 12\\
105 18\\
200 16\\
98 13\\
152 15\\
98 13\\
92 14\\
177 18\\
109 14\\
186 14\\
166 18\\
11 8\\
54 13\\
54 11\\
9 7\\
55 15\\
176 14\\
150 13\\
60 15\\
47 14\\
190 16\\
94 13\\
40 12\\
194 19\\
188 16\\
142 17\\
42 11\\
160 16\\
31 11\\
23 8\\
130 13\\
29 10\\
127 14\\
123 16\\
73 12\\
189 17\\
88 13\\
170 16\\
181 16\\
62 13\\
90 12\\
28 11\\
13 10\\
61 14\\
168 18\\
119 13\\
142 15\\
137 16\\
118 16\\
97 16\\
98 16\\
26 11\\
178 14\\
117 15\\
25 9\\
121 13\\
83 13\\
143 14\\
44 13\\
42 8\\
101 13\\
27 10\\
84 15\\
91 13\\
147 18\\
171 14\\
190 17\\
151 16\\
66 12\\
90 14\\
137 16\\
192 15\\
187 15\\
154 18\\
43 13\\
28 10\\
30 11\\
80 13\\
70 13\\
195 15\\
164 17\\
116 14\\
84 15\\
114 15\\
195 15\\
102 16\\
65 11\\
69 12\\
72 14\\
5 5\\
151 15\\
161 17\\
37 9\\
99 14\\
150 16\\
47 12\\
76 13\\
115 15\\
144 14\\
16 7\\
41 10\\
33 10\\
17 9\\
37 13\\
28 10\\
157 16\\
55 13\\
54 13\\
48 11\\
138 16\\
53 11\\
199 15\\
144 15\\
131 15\\
108 16\\
146 13\\
114 16\\
21 10\\
168 14\\
40 9\\
11 7\\
24 11\\
196 19\\
144 16\\
67 12\\
139 15\\
50 13\\
140 17\\
190 15\\
103 12\\
45 11\\
63 11\\
122 14\\
76 11\\
33 12\\
188 16\\
95 15\\
42 12\\
14 8\\
35 12\\
98 13\\
161 17\\
48 12\\
37 10\\
10 7\\
172 15\\
67 13\\
34 11\\
52 13\\
97 13\\
193 18\\
11 7\\
6 7\\
64 12\\
157 13\\
127 15\\
39 11\\
172 15\\
181 15\\
78 12\\
161 19\\
162 15\\
200 15\\
167 15\\
133 15\\
22 12\\
44 11\\
52 12\\
44 12\\
90 12\\
81 12\\
64 16\\
133 18\\
74 13\\
188 16\\
8 7\\
40 12\\
66 13\\
80 14\\
59 13\\
11 10\\
84 12\\
121 15\\
75 14\\
149 20\\
187 15\\
13 9\\
64 12\\
25 10\\
15 9\\
27 11\\
24 10\\
111 15\\
128 14\\
23 10\\
196 17\\
177 14\\
82 13\\
87 14\\
41 11\\
193 15\\
75 13\\
67 13\\
152 15\\
103 13\\
86 14\\
31 11\\
97 13\\
33 10\\
60 13\\
123 15\\
14 9\\
137 16\\
48 13\\
105 17\\
118 14\\
56 13\\
197 17\\
34 13\\
133 15\\
137 16\\
97 14\\
37 12\\
66 15\\
38 11\\
92 15\\
11 8\\
40 10\\
53 13\\
185 14\\
154 14\\
135 14\\
124 14\\
179 17\\
194 18\\
163 17\\
42 12\\
10 7\\
173 17\\
176 15\\
112 12\\
39 13\\
176 17\\
74 13\\
189 14\\
56 11\\
19 9\\
164 16\\
129 13\\
11 8\\
101 13\\
58 13\\
79 14\\
104 13\\
142 15\\
73 14\\
183 16\\
124 16\\
81 14\\
10 8\\
82 16\\
170 15\\
16 10\\
105 14\\
114 15\\
197 17\\
11 7\\
175 15\\
169 14\\
23 9\\
110 14\\
60 12\\
181 16\\
73 14\\
152 15\\
87 13\\
61 10\\
13 9\\
160 15\\
180 16\\
160 14\\
168 16\\
154 16\\
102 15\\
169 16\\
166 14\\
62 11\\
162 13\\
103 13\\
135 14\\
120 14\\
54 13\\
43 13\\
31 10\\
149 14\\
94 15\\
77 14\\
69 11\\
148 14\\
20 9\\
140 15\\
142 15\\
76 13\\
77 15\\
31 12\\
186 15\\
105 16\\
93 14\\
166 17\\
159 15\\
167 16\\
33 14\\
52 10\\
162 15\\
32 12\\
52 13\\
27 11\\
14 9\\
133 13\\
73 14\\
10 10\\
111 14\\
196 17\\
36 11\\
31 9\\
68 13\\
177 18\\
71 13\\
65 16\\
70 13\\
193 14\\
173 17\\
80 14\\
177 15\\
25 12\\
86 15\\
162 17\\
8 8\\
181 15\\
148 16\\
152 14\\
133 18\\
173 15\\
66 13\\
26 9\\
49 12\\
171 17\\
162 17\\
48 13\\
111 15\\
180 17\\
91 16\\
181 16\\
92 14\\
189 18\\
23 10\\
50 13\\
37 11\\
24 10\\
182 18\\
27 9\\
132 13\\
92 13\\
100 15\\
51 11\\
50 12\\
83 14\\
68 10\\
90 12\\
65 12\\
43 11\\
107 16\\
145 14\\
85 14\\
71 14\\
74 12\\
125 16\\
156 16\\
122 14\\
146 19\\
22 9\\
38 10\\
12 7\\
116 16\\
150 14\\
126 13\\
21 10\\
48 13\\
103 16\\
115 13\\
40 12\\
33 13\\
104 12\\
61 13\\
79 14\\
49 11\\
197 18\\
165 16\\
106 18\\
22 11\\
120 15\\
185 16\\
144 16\\
174 18\\
14 8\\
99 14\\
137 15\\
98 13\\
64 14\\
149 14\\
188 15\\
184 15\\
192 16\\
69 14\\
167 16\\
44 14\\
110 14\\
59 14\\
184 16\\
139 17\\
67 14\\
96 16\\
30 11\\
157 16\\
135 16\\
57 13\\
131 16\\
30 10\\
163 15\\
85 11\\
77 12\\
51 13\\
148 15\\
77 15\\
134 14\\
183 17\\
86 12\\
19 7\\
78 13\\
124 16\\
44 11\\
10 6\\
65 14\\
183 17\\
22 10\\
71 13\\
54 8\\
139 15\\
19 9\\
104 13\\
93 14\\
52 13\\
26 10\\
13 9\\
174 19\\
158 17\\
187 16\\
85 14\\
140 16\\
68 12\\
123 13\\
195 15\\
148 14\\
124 16\\
27 10\\
74 14\\
181 15\\
193 16\\
44 11\\
105 15\\
45 13\\
175 15\\
66 12\\
75 13\\
147 16\\
94 12\\
88 15\\
82 12\\
144 15\\
144 16\\
166 15\\
33 12\\
97 16\\
35 11\\
128 14\\
10 6\\
151 13\\
87 12\\
185 17\\
108 14\\
108 15\\
111 14\\
76 15\\
113 17\\
25 9\\
124 15\\
109 15\\
18 12\\
114 14\\
165 19\\
180 16\\
49 11\\
9 9\\
175 17\\
118 16\\
121 13\\
140 16\\
59 12\\
171 18\\
22 10\\
25 13\\
146 13\\
68 13\\
146 17\\
134 16\\
165 16\\
71 15\\
170 15\\
131 15\\
75 12\\
148 14\\
10 8\\
125 15\\
190 17\\
154 16\\
96 13\\
128 17\\
142 18\\
194 15\\
182 15\\
97 13\\
108 13\\
125 15\\
127 16\\
170 16\\
103 14\\
85 15\\
194 14\\
67 12\\
124 14\\
40 12\\
87 15\\
137 18\\
65 14\\
14 7\\
90 13\\
152 17\\
196 17\\
165 15\\
150 13\\
63 13\\
38 12\\
71 12\\
47 16\\
142 15\\
48 12\\
89 13\\
144 16\\
64 14\\
135 14\\
98 12\\
115 16\\
130 14\\
35 11\\
165 16\\
47 15\\
48 10\\
194 15\\
57 13\\
122 16\\
116 14\\
67 15\\
123 16\\
134 11\\
94 14\\
22 10\\
140 14\\
49 12\\
161 17\\
32 13\\
107 13\\
176 15\\
95 14\\
182 14\\
60 11\\
164 15\\
175 15\\
196 18\\
146 17\\
76 13\\
53 11\\
171 15\\
14 9\\
65 12\\
63 13\\
135 16\\
189 17\\
144 15\\
11 8\\
86 12\\
184 15\\
41 14\\
29 11\\
23 9\\
136 17\\
13 7\\
24 9\\
73 14\\
28 11\\
33 10\\
18 10\\
42 11\\
128 13\\
19 9\\
81 12\\
179 16\\
106 15\\
44 13\\
147 19\\
158 17\\
156 14\\
51 11\\
70 14\\
64 12\\
113 14\\
15 9\\
121 15\\
187 17\\
44 15\\
166 16\\
157 15\\
29 11\\
136 14\\
7 7\\
53 13\\
44 14\\
194 16\\
155 17\\
164 16\\
63 16\\
15 7\\
186 15\\
121 14\\
150 14\\
79 14\\
184 16\\
128 13\\
129 15\\
139 13\\
80 10\\
96 15\\
21 10\\
65 10\\
38 9\\
71 13\\
164 15\\
173 16\\
140 16\\
156 17\\
115 12\\
137 16\\
10 7\\
170 14\\
60 12\\
200 15\\
101 14\\
42 12\\
64 11\\
135 14\\
145 15\\
113 14\\
68 13\\
184 14\\
13 10\\
57 12\\
109 18\\
61 11\\
13 8\\
169 16\\
39 10\\
6 8\\
86 15\\
91 15\\
188 15\\
105 14\\
171 14\\
82 14\\
16 9\\
21 10\\
139 12\\
109 13\\
57 13\\
59 17\\
192 15\\
198 16\\
43 12\\
38 14\\
122 16\\
54 15\\
14 10\\
196 15\\
112 14\\
189 14\\
6 7\\
66 14\\
113 15\\
163 20\\
90 14\\
50 11\\
165 15\\
178 17\\
110 12\\
122 14\\
83 13\\
42 10\\
20 9\\
145 15\\
168 15\\
40 11\\
115 16\\
80 13\\
190 15\\
147 14\\
97 14\\
11 9\\
162 16\\
196 17\\
143 14\\
160 15\\
60 12\\
90 16\\
178 16\\
22 10\\
66 12\\
48 12\\
116 15\\
139 15\\
159 13\\
171 19\\
147 16\\
177 14\\
58 12\\
70 12\\
111 14\\
184 15\\
105 14\\
72 13\\
122 16\\
180 18\\
142 14\\
102 12\\
136 14\\
163 16\\
84 16\\
104 14\\
84 15\\
23 8\\
173 16\\
62 14\\
132 18\\
179 16\\
196 16\\
147 17\\
43 10\\
173 17\\
89 13\\
124 13\\
69 12\\
26 7\\
158 15\\
79 16\\
70 13\\
108 13\\
11 8\\
114 15\\
3 4\\
3 4\\
3 7\\
2 4\\
2 4\\
3 4\\
2 4\\
3 6\\
2 4\\
3 4\\
2 4\\
3 4\\
2 4\\
2 2\\
3 6\\
};
\end{axis}
\end{tikzpicture}%